\pdfoutput=1
\documentclass[12pt,reqno,a4wide]{amsart}

\usepackage[utf8]{inputenc}

\usepackage{tikz}
\usetikzlibrary{arrows,backgrounds,calc}
\usepackage{calrsfs}
\usepackage[charter]{mathdesign}
\usepackage{amsthm, amsmath, amsfonts}
\usepackage{graphicx}
\usepackage{hyperref}
\usepackage{lineno}
\usepackage{enumerate}
\usepackage{mathtools}
\usepackage{pgfplots}
\pgfplotsset{compat=1.5}

\usepackage[draft=false,kerning=true]{microtype}

\allowdisplaybreaks

\newtheorem{theorem}{Theorem}
\newtheorem{proposition}{Proposition}[section]
\newtheorem{corollary}[proposition]{Corollary}
\newtheorem{lemma}[proposition]{Lemma}

\theoremstyle{remark}
\newtheorem*{remark}{Remark}
\newtheorem*{notation}{Notation}

\theoremstyle{definition}
\newtheorem*{definition}{Definition}
\newtheorem{example}[proposition]{Example}

\renewcommand{\MR}[1]{}

\newcommand{\TODO}[1]%
{\par\fbox{\begin{minipage}{0.9\linewidth}\textbf{TODO:} #1\end{minipage}}\par}

\let\Re\relax

\DeclareSymbolFont{cmcal}{OMS}{cmsy}{m}{n}
\SetSymbolFont{cmcal}{bold}{OMS}{cmsy}{b}{n}
\DeclareSymbolFontAlphabet{\mathcal}{cmcal}

\newcommand{\phiop}{\mathop{\phi}\nolimits}
\newcommand{\PathGF}{\mathop{S}\nolimits}

\newcommand{\myBigl}{\biggl}
\newcommand{\myBigr}{\biggr}

\newcommand{\R}[0]{\mathbb{R}}
\newcommand{\Z}[0]{\mathbb{Z}}
\renewcommand{\P}[0]{\mathbb{P}}
\newcommand{\E}[0]{\mathbb{E}}
\newcommand{\V}[0]{\mathbb{V}}
\newcommand{\N}[0]{\mathbb{N}}
\DeclareMathOperator{\Re}{Re}

\DeclarePairedDelimiter{\abs}{\lvert}{\rvert}
\DeclarePairedDelimiter{\iverson}{\llbracket}{\rrbracket}

\title[]
{Ascents in non-negative Lattice Paths}

\author[B.~Hackl]{Benjamin Hackl}
\author[C.~Heuberger]{Clemens Heuberger}

\address[Benjamin Hackl, Clemens Heuberger]
{Institut f\"ur Mathematik,
  Alpen-Adria-Uni\-ver\-si\-t\"at Klagenfurt, Universit\"atsstra\ss e
  65--67, 9020 Klagenfurt, Austria}
\email{\href{mailto:benjamin.hackl@aau.at}{benjamin.hackl@aau.at}}
\email{\href{mailto:clemens.heuberger@aau.at}{clemens.heuberger@aau.at}}
\thanks{B.~Hackl and C.~Heuberger are supported by the Austrian
  Science Fund (FWF): P~24644-N26 and by the Karl Popper Kolleg
  ``Modeling-Simulation-Optimization'' funded by Alpen-Adria-Universit\"at Klagenfurt
  and by the Carinthian Economic Promotion Fund (KWF)}

\author[H.~Prodinger]{Helmut Prodinger}
\address[Helmut Prodinger]{Department of Mathematical
  Sciences, Stellenbosch University, 7602 Stellenbosch,
 South Africa}
\email{\href{mailto:hproding@sun.ac.za}{hproding@sun.ac.za}}
\thanks{H.~Prodinger is supported by an incentive grant of the
  National Research Foundation of South Africa}


\begin{document}

\maketitle
\begin{abstract}
  Non-negative Łukasiewicz paths are special two-dimensional lattice paths never passing
  below their starting altitude which have only one single special type of down
  step. They are well-known and -studied combinatorial objects, in particular
  due to their bijective relation to trees with given node degrees.

  We study the asymptotic behavior of the number of ascents (i.e., the number of maximal
  sequences of consecutive up steps) of given length for classical
  subfamilies of general non-negative Łukasiewicz paths: those with arbitrary ending altitude, those
  ending on their starting altitude, and a variation thereof. Our results include precise
  asymptotic expansions for the expected number of such ascents as well as for the
  corresponding variance.
\end{abstract}

\section{Introduction}\label{sec:introduction}

Two-dimensional lattice paths can be defined as sequences of points in the plane $\R^{2}$
where for any point, the vector pointing to the succeeding point (``step'') is from a
predefined finite set, the \emph{step set}. In general, lattice paths are very
classical combinatorial objects with a variety of applications in, amongst others,
Biology, Physics, and Chemistry.

In this paper, our focus lies on a special class of two-dimensional lattice paths:
\emph{non-negative simple \L ukasiewicz paths}. A lattice path is said to be \emph{simple}
if the horizontal coordinate is the same (e.g.\ is~1) for all
possible steps. In case of a simple path family, we define the \emph{step set}
$\mathcal{S}$ as the set of allowed height differences, i.e., the respective
$y$-coordinates between the points of the path. If, additionally, the step set
$\mathcal{S} \subseteq \Z$ is integer-valued and contains $-1$ as the single negative
value (meaning that all other values in $\mathcal{S}$ are non-negative), then the
corresponding paths are called \emph{simple \L ukasiewicz paths}.

If a lattice path starts at the origin and never passes below the horizontal axis, then
the path is said to be a \emph{meander} (or non-negative path). And in case such a
non-negative path ends on the horizontal axis, it is called an \emph{excursion}.

We are interested in analyzing the number of \emph{ascents} in these paths. An
\emph{ascent} is an inclusion-wise maximal sequence of up steps (i.e., steps in
$\mathcal{S}\setminus\{-1\}$; this might also include the horizontal step corresponding to
$0$). For an integer $r \geq 1$, if an ascent consists of precisely $r$
steps, then the ascent is said to be an $r$-ascent. As an example,
Figure~\ref{fig:paths-and-ascents} depicts some non-negative \L ukasiewicz excursion with
emphasized $2$-ascents.

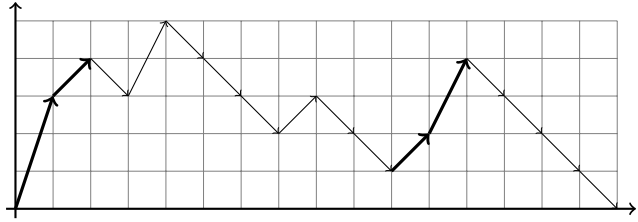
\begin{figure}[ht]
  \centering
  \begin{tikzpicture}[scale=0.5]
    \draw[help lines] (0,0) grid (16,5);
    \draw[thick, ->] (-0.25, 0) -- (16.5, 0);
    \draw[thick, ->] (0, -0.25) -- (0, 5.5);
    \draw[->, very thick] (0,0) -- (1, 3);
    \draw[->, very thick] (1,3) -- (2,4);
    \draw[->] (2,4) -- (3,3);
    \draw[->] (3,3) -- (4,5);
    \draw[->] (4,5) -- (5,4);
    \draw[->] (5,4) -- (6,3);
    \draw[->] (6,3) -- (7,2);
    \draw[->] (7,2) -- (8,3);
    \draw[->] (8,3) -- (9,2);
    \draw[->] (9,2) -- (10,1);
    \draw[->, very thick] (10,1) -- (11,2);
    \draw[->, very thick] (11,2) -- (12,4);
    \draw[->] (12,4) -- (13,3);
    \draw[->] (13,3) -- (14,2);
    \draw[->] (14,2) -- (15,1);
    \draw[->] (15,1) -- (16,0);
  \end{tikzpicture}
  \caption{Simple \L ukasiewicz excursion of length $16$ with emphasized $2$-ascents where
  $\mathcal{S} = \{-1, 1, 2, 3\}$}
  \label{fig:paths-and-ascents}
\end{figure}

In this paper, we give a precise analysis of the number of $r$-ascents for non-negative
simple \L ukasiewicz paths of given length, as well as of variants of this class of
lattice paths. Our investigation is motivated by
\cite{Kangro-Pourmoradnasseri-Theis:2016:ascents-dispersed-dyck}, where the number
of $1$-ascents in a special lattice path class related to the classic Dyck paths was analyzed
explicitly by elementary methods.

\subsection*{Main Results}
Within this paper, three special classes of non-negative \L ukasiewicz paths are of interest:
\begin{itemize}
\item \emph{excursions}, i.e., paths that end on the horizontal axis,
\item \emph{dispersed excursions}, i.e., excursions where horizontal steps are not allowed
  except on the horizontal axis,
\item \emph{meanders}, i.e., general non-negative \L ukasiewicz paths without additional
  restrictions.
\end{itemize}
Formally, we conduct our analysis by investigating random variables $E_{n,r}$, $D_{n,r}$,
$M_{n,r}$ which model the number of $r$-ascents in a random excursion, dispersed excursion, and
meander of length~$n$, respectively. The underlying probability models are based on 
equidistribution: within a family, all paths of length $n$ are assumed to be equally likely.

Given $r\in \N$ and considering $n\in\N_{0}$ with $n\to\infty$, we prove that for
excursions we have
\[ \E E_{n,r} = \mu n + c_{0} + O(n^{-1/2}) \qquad\text{ and }\qquad \V E_{n,r} =
  \sigma^{2}n + O(n^{1/2}),  \]
for some constants $\mu$, $c_{0}$, $\sigma^{2}$ depending on the chosen step set
$\mathcal{S}$. The constants are given explicitly in Theorem~\ref{thm:ascents:excursions}. 
Additionally, if $n$ is not a multiple of the so-called \emph{period} of
the step set, then the random variable degenerates and we have $E_{n,r} = 0$; see
Theorem~\ref{thm:ascents:excursions} for details.

For dispersed excursions, the corresponding computations get rather messy, which is why we
restrict ourselves to the investigation of $d_{n}$, the number of dispersed excursions of length~$n$,
as well as the expected value $\E D_{n,r}$. In particular, for all step sets
$\mathcal{S}$ (except for the special case of dispersed Dyck paths with $\mathcal{S} = \{-1,
1\}$), $d_{n}$ satisfies
\[ d_{n} = c_{0} \kappa^{n} n^{-3/2} + O(\kappa^{n} n^{-5/2}), \]
with constants $c_{0}$ and $\kappa$ depending on the chosen step set. For the
expected number of ascents in this particular lattice path family, we find
\[ \E D_{n,r} = \mu n + O(1)  \]
for some constant $\mu$ depending on $\mathcal{S}$. Explicit values for these constants
and more details are given in
Theorem~\ref{thm:dispersed:tau-not-1}.

In the context of meanders we are able to show that for all step sets (with two
special exceptions: Dyck meanders with $\mathcal{S} = \{-1, 1\}$, and Motzkin meanders with
$\mathcal{S} = \{-1, 0, 1\}$) we have
\[ \E M_{n,r} = \mu n + c_{0} + O(n^{5/2} \kappa^{n}) \qquad\text{ and }\qquad \V M_{n,r}
  = \sigma^{2}n + O(1), \]
for constants $\mu$, $c_{0}$, $\kappa\in (0,1)$, $\sigma^{2}$ depending on
$\mathcal{S}$. Also, the random variable $M_{n,r}$ is asymptotically normally
distributed; see Theorem~\ref{thm:meanders:tau-not-1} for explicit formulas for the
constants and more details.

In theory, our approach can be used to obtain arbitrarily precise asymptotic expansions
for all the quantities above. For the sake of readability we have chosen to only give the
main term as well as one additional term, wherever possible.

On a more technical note, in order to deal with general \L ukasiewicz step sets in our
setting, we make use of a
generating function approach (see~\cite{Flajolet-Sedgewick:ta:analy}). In particular, we
heavily rely on the technique of \emph{singular inversion}
\cite[Chapter VI.7]{Flajolet-Sedgewick:ta:analy}, which deals with finding an
asymptotic expansion for the growth of the coefficients of generating functions $y(z)$ satisfying
a functional equation of the type
\[ y = z \phiop(y) \]
with a suitable function $\phiop$.

\subsection*{Notation and Special Cases}
Throughout this paper, the step set will be denoted as $\mathcal{S} = \{-1, b_{1}, \dots,
b_{m-1}\}$ with integers $b_{j} \geq 0$ for all $j$ and $m\geq 1$. The $b_{j}$ are referred to as
up steps---even if the step is a horizontal one.

The so-called characteristic polynomial of the lattice path class,
i.e., the generating function corresponding to the set $\mathcal{S}$, is denoted by $\PathGF(u) :=
\sum_{s\in\mathcal{S}} u^{s}$. The strongly related generating function of the non-negative
steps is denoted by $\PathGF_{+}(u) := \sum_{\substack{s\in\mathcal{S}\\ s \geq 0}} u^{s}$.

In this context, observe that the particular step set $\mathcal{S} = \{-1, 0\}$ corresponds
to a, in some sense, pathological family of \L ukasiewicz paths. In this case, there is only
precisely one non-negative \L ukasiewicz path of any given length. The family of meanders
and excursions coincides, and also the random variables degenerate in the sense that we 
have\footnote{We make use of the Iversonian notation popularized 
in~\cite[Chapter 2]{Graham-Knuth-Patashnik:1994}: $\iverson{\mathit{expr}}$ takes value $1$
if $\mathit{expr}$ is true, and $0$ otherwise.}
$E_{n,r} = M_{n,r} = \iverson{n = r}$. Thus, further investigation of this case is not
required---which is why we exclude the case $\mathcal{S} = \{-1, 0\}$ from now on.

While in the case of a general step set $\mathcal{S}$ we are forced to deal with
implicitly given quantities, for special cases like $\mathcal{S} = \{-1, 1\}$ (Dyck
paths), everything can be made completely explicit as we will demonstrate in the course of
our investigations.

Finally, we make use of the following well-established notation: For a generating function
$f(z) = \sum_{n\geq 0} f_{n} z^{n}$, the coefficient belonging to $z^{n}$ is
denoted as $f_{n} = [z^{n}] f(z)$.

\subsection*{Structure of this Article}
In Sections~\ref{sec:gf-analytic} and~\ref{sec:gf-combinatorial}, we demonstrate two
different approaches to determine suitable generating functions required to analyze the
number of ascents. The approach in Section~\ref{sec:gf-analytic} is fully analytic and
fueled by the kernel method and the ``adding a new slice''-technique, see~\cite[Section
2.5]{Bona:Prodinger:2015:analyt}. The other approach in Section~\ref{sec:gf-combinatorial}
is a more combinatorial approach based on the inherent relation between \L ukasiewicz
paths and plane trees with given vertex degrees.
Formulas for the respective generating functions are given in
Propositions~\ref{prop:lukasiewicz-ascents} and~\ref{prop:gf-combinatorial}.

Then, in Section~\ref{sec:SA-inverse} we give a rigorous description of the singular
structure of a fundamental quantity, namely a particular inverse function derived in
Sections~\ref{sec:gf-analytic}
and~\ref{sec:gf-combinatorial}. Important tools for giving this description are provided
in Propositions~\ref{prop:singular-inversion-expansion}
and~\ref{prop:singularities-roots-of-unity}, which are extensions
of~\cite[Theorem VI.6; Remark VI.17]{Flajolet-Sedgewick:ta:analy}.

Section~\ref{sec:ascents} contains the actual analysis of ascents for the different
lattice path families mentioned above. In particular, in
Section~\ref{sec:ascents:excursions} we investigate excursions; the main result is stated
in Theorem~\ref{thm:ascents:excursions}. Section~\ref{sec:ascents:dispersed} deals with
the analysis of ascents in dispersed excursions. In this case, the expected number of
$r$-ascents for all but one given step sets is analyzed within
Theorem~\ref{thm:dispersed:tau-not-1}, and the analysis for the remaining one
is conducted in Proposition~\ref{prop:ascents-dispersed-dyck}.
Finally, Section~\ref{sec:ascents:meanders} contains our results for ascents in
meanders. Similarly to the previous section, the analysis for most step sets is given in
Theorem~\ref{thm:meanders:tau-not-1}, and the remaining cases are investigated in
Propositions~\ref{prop:meanders:dyck} and~\ref{prop:meanders:motzkin}.

Many of the calculations in this article are computationally
quite involved. To this end, we make use of the package for asymptotic
expansions~\cite{Hackl-Heuberger-Krenn:2016:asy-sagemath} contained in the open-source
mathematics software system
SageMath~\cite{SageMath:2017:8.1}. The worksheets used to produce our results can be found
at
\begin{center}
\url{https://benjamin-hackl.at/publications/lukasiewicz-ascents/},
\end{center}
in particular:
\begin{itemize}
\item \href{https://benjamin-hackl.at/downloads/lukasiewicz-ascents/lukasiewicz-excursions.ipynb}{\texttt{lukasiewicz-excursions.ipynb}} contains the calculations from Section~\ref{sec:ascents:excursions},
\item \href{https://benjamin-hackl.at/downloads/lukasiewicz-ascents/lukasiewicz-dispersed-excursions.ipynb}{\texttt{lukasiewicz-dispersed-excursions.ipynb}} for
  Section~\ref{sec:ascents:dispersed},
\item \href{https://benjamin-hackl.at/downloads/lukasiewicz-ascents/lukasiewicz-meanders.ipynb}{\texttt{lukasiewicz-meanders.ipynb}}
  for Section~\ref{sec:ascents:meanders}, and
\item
  \href{https://benjamin-hackl.at/downloads/lukasiewicz-ascents/utilities.py}{\texttt{utilities.py}},
  which has to be copied into the same folder as the others. This file contains utility
  code that is used in the notebook files.
\end{itemize}

\section{Generating Functions: An Analytic Approach}\label{sec:gf-analytic}
In this section we will introduce and discuss the preliminaries required in order to carry
out the asymptotic analysis of ascents in the different path classes. We begin by taking a
closer look at the structure of \L ukasiewicz paths.

Of course, the number of excursions of given length $n$ strongly depends on
the structure of the step set $\mathcal{S}$. For example, in the case of Dyck paths, i.e., $\mathcal{S} = \{-1,
1\}$, there cannot be any excursions of odd length---Dyck paths are said to be periodic
lattice paths.

\begin{definition}[Periodicity of lattice paths]
  Let $\mathcal{S}$ be a \L ukasiewicz step set with corresponding characteristic
  polynomial $\PathGF(u) = \sum_{s\in \mathcal{S}} u^{s}$. Then the period of $\mathcal{S}$ (and the
  associated lattice path family) is the largest integer $p$ for which a
  polynomial $Q$ satisfying
  \[ u\PathGF(u) = Q(u^{p})  \]
  exists. If $p = 1$, then $\mathcal{S}$ is said to be aperiodic, otherwise $\mathcal{S}$
  is said to be $p$-periodic.
\end{definition}
\begin{remark}
  Observe that if a step set $\mathcal{S}$ has period $p$, then there are only excursions
  of length $n$ where $n\equiv 0\pmod p$. This can be seen by considering the
  generating function enumerating unrestricted paths of length $n$ with respect to their
  height, i.e., $\PathGF(u)^{n}$. Obviously, the number of excursions of length $n$ is at most
  the number of unrestricted paths ending at altitude $0$, and the
  latter one can be written as
  \[ [u^{0}] \PathGF(u)^{n} = [u^{n}] (u \PathGF(u))^{n} = [u^{n}] Q(u^{p})^{n}.  \]
  Hence, if $n \not\equiv 0\pmod p$, there are no unrestricted paths ending on the
  horizontal axis---and thus also no excursions.
\end{remark}

With these elementary observations in mind, we can now focus on our main problem: determining a
suitable generating function in order to enumerate $r$-ascents in different classes of
nonnegative \L ukasiewicz paths. In this context, the well-known \emph{kernel method} will
prove to be an appropriate approach.

\begin{proposition}\label{prop:lukasiewicz-ascents}
  Let $F(z,t,v)$ be the trivariate ordinary generating function counting non-negative
  \L ukasiewicz paths with step set $\mathcal{S}$ starting at $0$, where $z$ marks the
  length of the path, $t$ marks the number of $r$-ascents, and $v$ marks the final
  altitude of the path. Then $F(z,t,v)$ can be expressed as
  \begin{equation}\label{eq:lukasiewicz-ascents:p-plus}
    F(z,t,v)  = \frac{v - V(z,t)}{v - z L(z,t,v)} L(z,t,v),
  \end{equation}
  where
  \[ L(z,t,v) = \frac{1}{1 - z\PathGF_{+}(v)} + (t-1) (z\PathGF_{+}(v))^{r},  \]
  and where $v = V(z,t)$ is the unique solution of the polynomial equation
  \begin{equation}\label{eq:lukasiewicz-v-equation}
    \big(v - z - z(t-1)(z\PathGF_{+}(v))^{r}\big)(1 - z\PathGF_{+}(v)) - z^{2}\PathGF_{+}(v) = 0
  \end{equation}
  that satisfies $V(0,1) = 0$. The function $V(z,t)$ is holomorphic in a neighborhood of
  $(z,t) = (0,1)$.
\end{proposition}
\begin{proof}
  Let $\Phi^{(k)}(z,t,v)$ denote the trivariate generating function enumerating
  non-negative \L ukasiewicz paths with respect to the step set $\mathcal{S}$
  with precisely $k$ mountains (i.e., with $k$ occurrences of
  the pattern $\nearrow\searrow$, where $\nearrow$ represents any of the allowed
  up steps) that end in a down step. The variables $z$, $t$, and $v$ mark path
  length, number of $r$-ascents, and the final altitude of the path,
  respectively.

  By definition of $\Phi^{(k)}(z,t,v)$, we have $\Phi^{(0)}(z,t,v) = 1$. We now construct
  $\Phi^{(k+1)}(z,t,v)$ from $\Phi^{(k)}(z,t,v)$ in order to establish a functional
  identity. Observe that an $r$-ascent occurs precisely if a sequence of $r$ upsteps
  followed by at least one downstep is added after a downstep. Then, observe that
  $L(z,t,v)$ as defined in the statement above
  is the generating function corresponding to a sequence of up steps (the first
  summand enumerates sequences of up steps without considering $r$-ascents; the
  second summand marks sequences of length $r$ with the variable $t$). Thus, $L(z,t,v) -
  1$ enumerates non-empty sequences of up steps. From there it
  is easy to see that the generating function $\Phi^{(k)}(z,t,v) (L(z,t,v) - 1)$
  enumerates paths with $k$ mountains and an appended non-empty sequence of up steps. The
  generating function $\Phi^{(k+1)}(z,t,v)$ can then be obtained by attaching another non-empty sequence of
  down steps. Formally, this can be achieved by substituting $v^{j}$ (which corresponds to
  a path ending at altitude $j$) with
  \[ v^{j} \mapsto \sum_{\ell=1}^{j} z^{\ell} v^{j-\ell} = \frac{z/v}{1 - z/v} (v^{j} -
    z^{j}).  \]
  In particular, this substitution is also applied for $j=0$, i.e., the paths
  with altitude $0$ (which correspond to $v^{0}$). In this case the substitution reads
  $v^{0} \mapsto 0$, which means that these paths get eliminated as we are
  not allowed to take a subsequent down step.
  Altogether, this gives the recurrence relation
  \begin{multline}\label{eq:phi-recurrence}
    \Phi^{(k+1)}(z,t,v) = \frac{z/v}{1 - z/v} \big(\Phi^{(k)}(z,t,v)(L(z,t,v) - 1) -
    \Phi^{(k)}(z,t,z)(L(z,t,z) - 1)\big),
  \end{multline}
  because carrying out the substitution gives the difference
  between $\Phi^{(k)}(z,t,v) (L(z,t,v) - 1)$ and the same term with $v$ replaced by $z$,
  multiplied with the factor $\frac{z/v}{1 - z/v}$.
  Observe that summing $\Phi^{(k)}(z,t,v)$ over $k \geq 0$ yields the generating function
  \[ \Phi(z,t,v) := \sum_{k\geq 0} \Phi^{(k)}(z,t,v)  \]
  enumerating paths with an arbitrary number of mountains that end on a down step. Thus, summation over $k\geq 0$
  in~\eqref{eq:phi-recurrence} proves that $\Phi(z,t,v)$ satisfies the functional equation
  \begin{equation}\label{eq:phi-functional}
    \Phi(z,t,v) \big(1 - \frac{z}{v-z} (L(z,t,v) - 1)\big) =
    1 - \frac{z}{v-z} \Phi(z,t,z)(L(z,t,z) - 1),
  \end{equation}
  where the summand $1$ on the right-hand side actually comes from $\Phi^{(0)}(z,t,v) = 1$.

  Rewriting the left-hand side of this equation after plugging in the definition of $L(z,t,v)$ yields
  \[ \Phi(z,t,v) \frac{(v-z - z(t-1)(z\PathGF_{+}(v))^{r})(1 - z\PathGF_{+}(v)) -
      z^{2}\PathGF_{+}(v)}{(v-z)(1 - z\PathGF_{+}(v))}.  \]
  We proceed in the spirit of the kernel method
  (see~\cite{Banderier-et-al:2002:generating-functions-trees,
    Prodinger:2003:kernel-examples}), which basically revolves
  around the idea of setting $v$ to a suitable root $V(z,t)$ of the polynomial in
  the numerator such that the numerator~\eqref{eq:lukasiewicz-v-equation} of the left-hand side disappears.

  The existence of such a function is guaranteed by means of the holomorphic implicit
  function theorem (see, e.g., \cite[Section
  0.8]{Kaup-Kaup:1983:holomorphic-several-variables}). In our case this theorem allows to
  conclude that in a sufficiently small neighborhood of $(z,t) = (0,1)$ there has to be a
  unique holomorphic function $V(z,t)$ satisfying $V(0,1) = 0$ such that the numerator in
  the left-hand side of~\eqref{eq:phi-functional} disappears by setting $v = V(z,t)$.

  At the same time, the denominator cannot vanish: For $t = 1$
  Equation~\eqref{eq:lukasiewicz-v-equation} defining $V(z,t)$ can be rewritten as $V(z,1)
  = z V(z,1) \PathGF(V(z,1))$, which allows us to obtain a power series expansion for $V(z,1)$ in the
  neighborhood of $z = 0$. As $V(z,t)$ is holomorphic and $V(z,1) \neq z$, we find $V(z,t)
  \neq z$ in a small neighborhood of $(z,t) = (0,1)$ by continuity. Thus, the first factor
  in the denominator does not vanish there. Simultaneously, concerning the second factor, if we had
  $z\PathGF_{+}(V(z,t)) = 1$, then the kernel equation~\eqref{eq:lukasiewicz-v-equation} would
  simplify to $-z$ (which is not identically $0$ in a neighborhood of $z=0$), meaning that $v = V(z,t)$ would no longer be a solution, in
  contradiction to its definition. Thus, the denominator does not vanish simultaneously
  with the numerator in a small neighborhood of $(z,t) = (0,1)$.

  Thus, we can use this implicitly defined function to find an expression for
  $\Phi(z,t,z)$---and then, after plugging that into~\eqref{eq:phi-functional} and
  doing some simplifications, we arrive at
  \[ \Phi(z,t,v) = \frac{v - V(z,t)}{v - z L(z,t,v)}. \]
  Then, in order to prove~\eqref{eq:lukasiewicz-ascents:p-plus}, recall that $\Phi(z,t,v)$
  enumerates all non-negative \L ukasiewicz paths ending on a down step
  $\searrow$. Thus, the generating function enumerating all non-negative \L ukasiewicz
  paths $F(z,t,v)$ can be obtained from $\Phi(z,t,v)$ by appending another
  (possibly empty) sequence of upsteps, and accounting for another possible $r$-ascent. This yields
  \[  F(z,t,v) = \Phi(z,t,v) L(z,t,v) \]
  and proves the statement.
\end{proof}

The combinatorial nature of $F(z,t,v)$ allows us to draw an interesting conclusion
with respect to the implicitly defined function $V(z,t)$.

\begin{corollary}\label{cor:V-combinatorial-interpretation}
  Let $V(z,t)$ be the implicitly defined function
  solving~\eqref{eq:lukasiewicz-v-equation} from
  Proposition~\ref{prop:lukasiewicz-ascents}. Assume that the underlying step set
  $\mathcal{S}$ has period $p$. Then $V(z,t)$ is analytic around the origin
  $(z,t) = (0,0)$ with power series representation
  \begin{equation}\label{eq:V-power-series}
    V(z,t) = \sum_{j\geq 0} g_{j}(t) z^{jp+1},
  \end{equation}
  where the $g_{j}(t)$ are polynomials with integer coefficients. Combinatorially,
  $V(z,t)/z$ is the bivariate generating function enumerating \L ukasiewicz excursions with
  respect to $\mathcal{S}$, where $z$ and $t$ mark the length of the path and the number
  of $r$-ascents within, respectively.
\end{corollary}
\begin{proof}
  Setting $v = 0$, i.e., ignoring all \L ukasiewicz paths not ending on the starting
  altitude, yields $F(z,t,0) = V(z,t)/z$ as the factor $L(z,t,v)$ in~\eqref{eq:lukasiewicz-ascents:p-plus} then cancels against
  the denominator. The combinatorial interpretation of the
  trivariate generating function $F(z,t,v)$ together with the fact that for a
  $p$-periodic step set $\mathcal{S}$ there are no \L ukasiewicz excursions of length $n$
  for $p\nmid n$ proves all the statements above.
\end{proof}

Now, with an appropriate generating function at hand let us discuss our approach
for the asymptotic analysis of the number of ascents in a nutshell.

Basically, we set $v=0$ to obtain a bivariate generating function enumerating \L
ukasiewicz excursions, and we set $v = 1$ to obtain a generating function enumerating \L
ukasiewicz meanders. The appropriate generating functions for the factorial moments of $E_{n,r}$
and $M_{n,r}$ (from which expected value and variance can be computed) are then obtained
by first differentiating the corresponding generating function with respect to $t$
(possibly more often than once) and then setting $t=1$ in this partial derivative. The
growth of the coefficients of this
function can then be extracted by means of singularity analysis.

In particular, this means that in order to compute the asymptotic expansions for the
quantities we are interested in, we only need more information on $V(z,1)$ as well as
the partial derivatives $\frac{\partial^{\nu}}{\partial t^{\nu}}V(z,t)\big|_{t=1}$.
\begin{notation}
  For the sake of simplicity, and because we will deal with these expressions throughout
  the entire paper, we omit the second argument in $V(z,t)$ in case $t = 1$, i.e., we set $V(z)
  := V(z,1)$, $V_{t}(z) := V_{t}(z,1) = \frac{\partial}{\partial t} V(z,t)|_{t=1}$, $V_{z}(z) := V_{z}(z,1) = \frac{\partial}{\partial z} V(z,t)|_{t=1}$, and so on.
\end{notation}

\begin{example}[Explicit $F(z,t,v)$]
  In the case of $\mathcal{S} = \{-1, 1\}$ and $r = 1$ the generating function
  $F(z,t,v)$ can be computed explicitly and we find
  \begin{multline}\label{eq:F-explicit-dyck}
    F(z,t,v) \\ = \frac{(1 + (t-1)vz(1-vz))((1 - 2vz)(1 - (t-1)z^{2}) - \sqrt{(1 -
        (t+3)z^{2})(1 - (t-1)z^{2})}}{2z(1 - (t-1)z^{2})(z - v + v^{2}z +
      vz^{2}(t-1)(1-z))}.
  \end{multline}
  In particular, for $v = 0$ (i.e., when we want to get the generating function for
  $1$-ascents in Dyck paths), we find
  \[ F(z,t,0) = \frac{1 - (t-1)z^{2} - \sqrt{(1 - (t+3)z^{2})(1 - (t-1)z^{2})}}{2z^{2}
    (1 - (t-1)z^{2})}.  \]
\end{example}

\section{Generating Functions: A Combinatorial Approach}\label{sec:gf-combinatorial}

The combinatorial interpretation of the implicitly defined function $V(z,t)$ solving the
kernel equation~\eqref{eq:lukasiewicz-v-equation} motivates the question whether there is a construction of $F(z,t,v)$ that is
derived from the underlying combinatorial structure, instead of finding this structure as
a side effect.

The following proposition describes an integral relation enabling a purely combinatorial
derivation of the generating function $F(z,t,v)$.

\begin{proposition}\label{prop:lukasiewicz-bijection}
  The excursions of \L ukasiewicz paths of length $n$ with respect to some step set
  $\mathcal{S}$ correspond to rooted plane
  trees with $n+1$ nodes and node degrees contained in the set $1 + \mathcal{S}$.

  An $r$-ascent in a \L ukasiewicz excursion with respect to the step set $\mathcal{S}$ corresponds to a rooted
  subtree such that the leftmost leaf in this subtree has height $r$, and additionally the
  root node of the subtree is not a leftmost child itself (in the original tree).
\end{proposition}
\begin{proof}
  As pointed out in e.g.~\cite[Example 3]{Banderier-Flajolet:2002:lat-path}, this
  bijection between rooted plane trees with given node degrees and \L ukasiewicz
  excursions is well known. See~\cite[Section 11.3]{Lothaire:1997:comb-words} for an
  approach using words. However, as this bijection and its consequences makes up an
  integral part of the argumentation within this paper, we present a short proof
  ourselves. Furthermore, proving the bijection allows us to find the substructure in the
  tree corresponding to an $r$-ascent.

  Given a rooted plane tree $T$ consisting of $n$ nodes whose outdegrees are contained in $1 +
  \mathcal{S}$, we construct a lattice path as follows: when traversing the tree in
  preorder\footnote{Traversing a tree in preorder corresponds to the order in which the
    nodes are visited when carrying out a depth-first search on it.}, if passing a node
  with outdegree $d$, take a step of height $d-1$. The
  resulting lattice path thus consists of $n$ steps, and always ends on altitude $-1$,
  which follows from
  \[ \sum_{v\in T} (\deg(v) - 1) = \sum_{v\in T} \deg(v) - n = (n-1) - n = -1, \]
  where $\deg(v)$ denotes the outdegree (i.e., the number of children) of a node $v$ in the tree $T$.
  In particular, observe that by taking the first $n-1$ steps of the lattice path, we
  actually end up with a \L ukasiewicz excursion using the steps from $\mathcal{S}$. To
  see this, first observe that as the last node traversed in preorder certainly is a leaf,
  meaning that the $n$th step in the corresponding lattice path is a down step. As the
  path ends on altitude $-1$ after $n$ steps, we have to arrive at the starting altitude
  after $n-1$ steps.

  Furthermore, as illustrated in
  Figure~\ref{fig:lukasiewicz-bijection}, adding one to the current height of the
  constructed lattice path gives the size of the stack remembering the children that still
  have to be visited while traversing the tree in preorder. Combining the two previous
  arguments proves that the first $n-1$ steps in the constructed lattice path form a
  \L ukasiewicz excursion.

  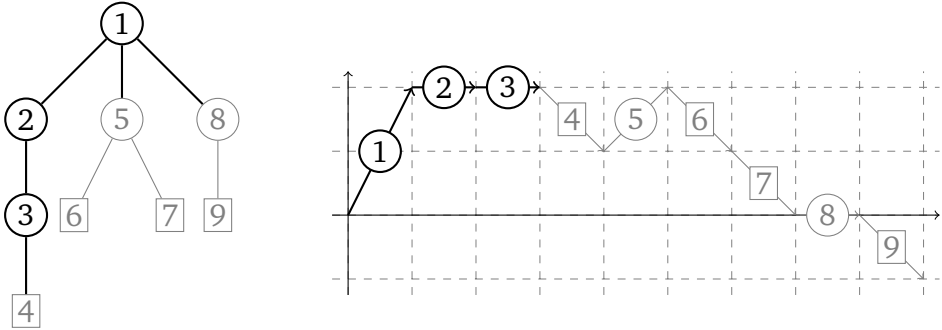
\begin{figure}[ht]
    \centering
    \begin{tikzpicture}[scale=0.85, inner sep=2pt]
      \node[draw, circle, thick] {1}
      child[thick] {node[draw, circle] {2} child {node[draw, circle] {3}
          child {node[draw, rectangle, thin, gray] {4}}}}
      child[thick] {node[draw, circle, thin, gray] {5}
        child[thin, gray] {node[draw, rectangle] {6}}
        child[thin, gray] {node[draw, rectangle] {7}}
      }
      child[thick] {node[draw, circle, thin, gray] {8} child[thin, gray] {node[draw, rectangle] {9}}};
    \end{tikzpicture}
    \hspace{1cm}
    \begin{tikzpicture}[scale=0.85, baseline=-1.5cm, inner sep=2pt]
      \draw[help lines, very thin, dashed] (-0.25, -1.25) grid (9.25, 2.25);
      \draw[->] (-0.25,0) -- (9.25, 0);
      \draw[->] (0, -1.25) -- (0, 2.25);
      \draw[->, thick] (0,0) to node[fill=white, draw, circle] {1} (1,2);
      \draw[->, thick] (1,2) to node[fill=white, draw, circle] {2} (2,2);
      \draw[->, thick] (2,2) to node[fill=white, draw, circle] {3} (3,2);
      \draw[->, gray] (3,2) to node[fill=white, draw, rectangle] {4} (4,1);
      \draw[->, gray] (4,1) to node[fill=white, draw, circle] {5} (5,2);
      \draw[->, gray] (5,2) to node[fill=white, draw, rectangle] {6} (6,1);
      \draw[->, gray] (6,1) to node[fill=white, draw, rectangle] {7} (7,0);
      \draw[->, gray] (7,0) to node[fill=white, draw, circle] {8} (8,0);
      \draw[->, gray] (8,0) to node[fill=white, draw, rectangle] {9} (9, -1);
    \end{tikzpicture}
    \caption[Bijection between \L ukasiewicz paths and trees with given node
    degrees]{Bijection between \L ukasiewicz paths and trees with given node degrees. The
      emphasized nodes and edges indicate the construction of the tree after the first three
    steps, which illustrates that the height of the \L ukasiewicz path is one less than
    the number of available node positions in the tree.}
    \label{fig:lukasiewicz-bijection}
  \end{figure}

  Similarly, by simply reversing the lattice path construction, a rooted plane tree of
  size $n+1$ with node degrees in $1 + \mathcal{S}$ can be constructed from any \L
  ukasiewicz excursion of length $n$ with respect to $\mathcal{S}$. This establishes the bijection between the two
  combinatorial families.

  Finally, Figure~\ref{fig:lukasiewicz-bijection-ascents} illustrates what $r$-ascents in
  \L ukasiewicz paths are mapped to by means of the bijection above.
  \begin{figure}[ht]
    \centering
    \begin{tikzpicture}[scale=0.6, inner sep=3pt,
      level 1/.style={sibling distance=3.25cm},
      level 2/.style={sibling distance=1.25cm},
      gray]
      \node[draw, circle] {}
      child[thick, black] {node[draw, circle] {}
        child {node[draw, rectangle] {}}
        child[thin, gray] {node[draw, circle] {}
          child[thick, black] {node[draw, circle] {}
            child {node[draw, rectangle] {}}
            child[thin, gray] {node[draw, rectangle] {}}
          }
        }
      }
      child {node[draw, circle] {}
        child[thick, black] {node[draw, circle] {}
          child {node[draw, rectangle] {}}
        }
        child {node[draw, rectangle] {}}
        child {node[draw, rectangle] {}}
      }
      child {node[draw, circle] {}
        child {node[draw, circle] {}
          child {node[draw, circle] {}
            child {node[draw, rectangle] {}}
            child {node[draw, circle] {}
              child[thick, black] {node[draw, circle] {}
                child {node[draw, rectangle] {}}
                child[thin, gray] {node[draw, rectangle] {}}
                child[thin, gray] {node[draw, rectangle] {}}
              }
            }
            child {node[draw, rectangle] {}}
            child {node[draw, circle] {}
              child[thick, black] {node[draw, circle] {}
                child {node [draw, rectangle] {}}
              }
              child {node[draw, rectangle] {}}
            }
          }
        }
      }
      child {node[draw, circle] {}
        child[thick, black] {node[draw, circle] {}
          child {node[draw, rectangle] {}}
          child[thin, gray] {node[draw, rectangle] {}}
        }
      };
    \end{tikzpicture}
    \caption[Emphasized $2$-ascents in a plane tree with 30 nodes]{Plane tree with $30$
      nodes bijective to some \L ukasiewicz excursion with respect to the step set
      $\mathcal{S} = \{-1, 0, 1, 2, 3\}$ whose number of $2$-ascents is $6$. The edges and
      nodes corresponding to the $2$-ascents are emphasized.}
    \label{fig:lukasiewicz-bijection-ascents}
  \end{figure}
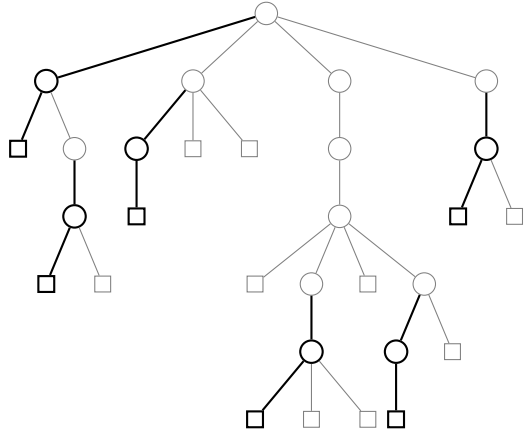
\end{proof}

In some sense, the bijection from Proposition~\ref{prop:lukasiewicz-bijection} can be seen as a generalization of the
well-known bijection between Dyck paths and binary trees where the tree is traversed in
preorder, internal nodes correspond to up steps and leaves to down steps.

The fact that there is this bijection between \L ukasiewicz excursions and these special
trees with given node degrees allows us to draw an immediate conclusion regarding the
corresponding generating functions.

\begin{corollary}\label{cor:lukasiewicz-gf-implicit}
  Let $V(z,t)$ be the generating function enumerating rooted plane trees with node degrees
  in $1 + \mathcal{S}$ where $z$ marks the number of nodes and $t$ marks the number of
  $r$-ascents in the corresponding \L ukasiewicz excursion. Then $V(z,t)/z$ enumerates
  \L ukasiewicz excursions with respect to $\mathcal{S}$ based on their length (marked by $z$) and the number of
  $r$-ascents (marked by $t$).

  Additionally, $V(z,t)$ satisfies the equations
  \begin{equation}\label{eq:V-combinatorial-func}
    V(0,t) = 0,\qquad V(z,t) = z L(z,t,V(z,t)),
  \end{equation}
  where
  \[ L(z,t,v) = \frac{1}{1 - z\PathGF_{+}(v)} + (t-1) (z\PathGF_{+}(v))^{r} \]
  enumerates sequences of up steps.
  In particular, $V(z) \coloneqq V(z,1)$, the ordinary generating function enumerating
  plane trees with node degrees in $1 + \mathcal{S}$ with respect to their size, satisfies
  \begin{equation}\label{eq:lukasiewicz-gf-implicit}
    V(0) = 0,\qquad V(z) = z V(z)\, \PathGF(V(z)).
  \end{equation}
\end{corollary}
\begin{proof}
  The first part of this statement is an immediate consequence of the bijection
  from Proposition~\ref{prop:lukasiewicz-bijection}. In order to
  prove~\eqref{eq:V-combinatorial-func}, we observe that $\mathcal{V}$, the 
  combinatorial class of plane trees with vertex outdegrees in $1 + \mathcal{S}$, can be
  constructed combinatorially by means of the symbolic equation
  \[ \mathcal{V} = \tikz{\node[draw, circle, inner sep=3pt, baseline=0] {};} \times \operatorname{SEQ}\Bigl(\tikz{\node[draw, circle, inner sep=3pt,
      baseline=0] {};} \times \sum_{\substack{s\in \mathcal{S}\\ s \geq 0}}
    \mathcal{V}^{s} \Bigr). \]
  In a nutshell, this constructs trees in $\mathcal{V}$ by explicitly building the path to
  the leftmost leaf (the first factor in the equation above) in the tree as a sequence of
  nodes. Apart from a leftmost child, these nodes also
  have an additional $s \in \mathcal{S}$ branches, $s\geq 0$, where again a tree from
  $\mathcal{V}$ is attached. Considering that we obtain an $r$-ascent when using this
  construction with a sequence of length $r$, this is precisely what is enumerated by
  $L(z,t,V(z,t))$. Thus, the symbolic equation directly translates
  into the functional equation in~\eqref{eq:V-combinatorial-func}. The condition $V(0,t) =
  0$ is a consequence of the fact that there are no rooted trees without nodes.

  Setting $t = 1$ in~\eqref{eq:V-combinatorial-func} leads
  to~\eqref{eq:lukasiewicz-gf-implicit}. We also want to give a combinatorial proof
  of~\eqref{eq:lukasiewicz-gf-implicit}:

  \begin{figure}[ht]
    \centering
        \begin{tikzpicture}[scale=0.75]
      \node (add) {$\mathcal{V}\quad=\quad\sum\limits_{s\in\mathcal{S}}$};
      \node[right of=add, circle, draw,  xshift=4em, yshift=2em, inner sep=4pt] (right-V) {};
      \node[below of=right-V, yshift=-1.5em, xshift=-3em] (2) {$\mathcal{V}$};
      \node[below of=right-V, yshift=-1.5em, xshift=-1em]  (3) {$\mathcal{V}$};
      \node[below of=right-V, yshift=-1.5em, xshift=1em, gray] (4) {$\cdots$};
      \node[below of=right-V, yshift=-1.5em, xshift=3em] (5) {$\mathcal{V}$};
      \draw (right-V) -- (2) (right-V) -- (3) (right-V) -- (5);
      \draw[dotted, gray] (right-V) -- (4);
      \draw [thick, decoration={brace, mirror, raise=1em},
             decorate] (2.center) to node (h) {} (5.center);
      \node [below of=h] {$1+s$};
    \end{tikzpicture}
    \caption[Symbolic equation for the family of plane trees with given
    outdegrees]{Symbolic equation for the family of plane trees $\mathcal{V}$ with
      outdegrees in $1 + \mathcal{S}$. The generating function for $\mathcal{V}$ is
      $V(z)$, and the root node is enumerated by $z$.}
    \label{fig:functional-equation}
  \end{figure}

  The implicit equation follows from the observation that a tree with node degrees
  from $1 + \mathcal{S}$ can be seen as a root node (enumerated by $z$) where $1 + s$ for
  $s\in \mathcal{S}$ such trees are attached. Translating this into the language of
  generating functions via the symbolic equation illustrated in Figure~\ref{fig:functional-equation}, yields
  \[ V(z) = z \sum_{s\in \mathcal{S}} V(z)^{1 + s} = z V(z) \PathGF(V(z)).\qedhere \]
\end{proof}

The shape of the functional equation~\eqref{eq:lukasiewicz-gf-implicit}, which is an
immediate consequence of the recursive structure of the underlying trees, is rather
special. While it is tempting to cancel $V(z)$ on both sides of this equation, it is
better to leave it in the present form: on the one hand, $\PathGF(u)$ starts with the summand
$1/u$---and on the other hand, we require~\eqref{eq:lukasiewicz-gf-implicit} to be in this
special form $y = z \phiop(y)$ such that we can use singular inversion to obtain
the asymptotic behavior of the coefficients of the generating function $V(z)$. This is
investigated in detail in Section~\ref{sec:SA-inverse}.

The following proposition is the combinatorial counterpart to
Proposition~\ref{prop:lukasiewicz-ascents}.

\begin{proposition}\label{prop:gf-combinatorial}
  Let $F(z,t,v)$ be the trivariate ordinary generating function counting non-negative
  \L ukasiewicz paths with step set $\mathcal{S}$ starting at $0$, where $z$ marks the
  length of the path, $t$ marks the number of $r$-ascents, and $v$ marks the final altitude
  of the path. Then $F(z,t,v)$ can be expressed as
  \begin{equation*}
    F(z,t,v) = \frac{v - V(z,t)}{v - z L(z,t,v)} L(z,t,v),
  \end{equation*}
  where $V(z,t)$ and $L(z,t,v)$ are defined as in
  Corollary~\ref{cor:lukasiewicz-gf-implicit}. 
\end{proposition}
\begin{proof}
  It is not hard to see that by considering a sequence of paths enumerated by $L(z,t,v)$ followed by a single
  down step (the corresponding 
  generating function for this class is $\frac{1}{1 - L(z,t,v)\, z/v}$), any unrestricted
  \L ukasiewicz path with respect to $\mathcal{S}$ ending on a down step can be constructed.

  We want to subtract all paths that pass below the starting altitude in order to obtain
  the trivariate generating function $\Phi(z,t,v)$ enumerating just the non-negative
  \L ukasiewicz paths. The paths passing below the axis can be decomposed into an
  excursion enumerated by $V(z,t)/z$ (see Corollary~\ref{cor:lukasiewicz-gf-implicit}),
  followed by an (illegal) down step enumerated by $z/v$, and ending with an unrestricted
  path again. Thus, the paths to be subtracted are enumerated by
  \[ \frac{V(z,t)}{z} \frac{z}{v} \frac{1}{1 - L(z,t,v)\frac{z}{v}}. \]
  Therefore, we find
  \[ \Phi(z,t,v) = \frac{v - V(z,t)}{v - zL(z,t,v)}.  \]
  Keeping in mind that $\Phi(z,t,v)$ only enumerates those non-negative \L ukasiewicz paths
  ending on a down step $\searrow$, the generating function $F(z,t,v)$ enumerating all
  such paths
  can be obtained by appending another sequence of upsteps, i.e.,
  \[ F(z,t,v) = \Phi(z,t,v) L(z,t,v).  \]
  This proves the statement.\qedhere
\end{proof}

Now, as we have derived a suitable generating function both via an analytic as well as via
a combinatorial approach, we are interested in extracting information like, for example, asymptotic
growth rates from
$F(z,t,v)$. In order to do so, we need to have a closer look at the function $V(z,t)$,
which, as we have already seen in both of the previous approaches, plays a fundamental
role in the analysis of ascents.

\section{Singularity Analysis of Inverse Functions}\label{sec:SA-inverse}
The aim of this section is, on the one hand, to state and prove an extension of
\cite[Remark~VI.17]{Flajolet-Sedgewick:ta:analy}. In fact, we simply confirm
what is announced in the footnote in
\cite[p.~405]{Flajolet-Sedgewick:ta:analy} and give more details. Then, we use these
results in order to derive relevant information on the generating function $V(z,t)$ from
before.

For the following two propositions, we borrow the notation used in~\cite[Chapter
VI.7]{Flajolet-Sedgewick:ta:analy}.

\begin{proposition}\label{prop:singular-inversion-expansion}
  Let $\phiop(u)$ be analytic with radius of convergence $0<R\le\infty$,
  $\phiop(0)\neq 0$, $[u^n]\phiop(u)\ge 0$ for all $n\ge 0$ and
  $\phiop(u)$ not affine linear. Assume that there is a positive $\tau\in (0,
  R)$ such that $\tau\phiop'(\tau)=\phiop(\tau)$. Finally assume that
  $\phiop(u)$ is a $p$-periodic power series for some maximal $p$. Denote the set of all $p$th
  roots of unity by $G(p)$.

  Then there is a unique function $y(z)$ satisfying $y(z) = z \phiop(y(z))$ which is
  analytic in a neighborhood of
  $0$ with $y(0)=0$. It has radius of convergence $\rho=\tau/\phiop(\tau)$
  around the origin. For $\abs{z}\le \rho$, it has exactly singularities at
  $z=\rho\zeta$ for $\zeta\in G(p)$. For $z\to \rho$, we have the singular
  expansion
  \begin{equation*}
    y(z) \stackrel{z\to\rho}{\sim} \sum_{j\ge 0} (-1)^j d_j\myBigl(1-\frac{z}{\rho}\myBigr)^{j/2}
  \end{equation*}
  for some computable constants $d_j$, $j\ge 0$. We have $d_0=\tau$ and
  $d_1=\sqrt{2\phiop(\tau)/\phiop''(\tau)}$. Additionally, we have $[z^n]y(z)=0$ for
  $n\not\equiv 1\pmod p$.
\end{proposition}

\begin{proof}
  Existence, uniqueness, radius of convergence as well as singular expansion
  around $z\to\rho$ of $y(z)$ are shown in
  \cite[Theorem~VI.6]{Flajolet-Sedgewick:ta:analy}.

  As $\phiop$ is a $p$-periodic power series and $\phiop(0)\neq 0$, there exists an aperiodic
  function $\chi$ such that $\phiop(u)=\chi(u^p)$. From the non-negativity of the
  coefficients of $\phiop(u)$, it is clear that $\chi(u)$ has
  non-negative coefficients and is analytic for $\abs{u}<R^p$. We consider
  $\psi(u)\coloneqq \chi(u)^p$. Then $\psi$ is again analytic for
  $\abs{u}<R^p$, it has clearly non-negative coefficients, $\psi(0)\neq 0$ and
  $\psi(u)$ is not an affine linear function. If $[u^m]\chi(u)>0$ and $[u^n]\chi(u)>0$
  for some $m<n$, then $[u^{pm}]\psi(u)>0$ as well as
  $[u^{pm+(n-m)}]\psi(u)>0$, which implies that $\psi$ is aperiodic.

  Finally, we have
  \begin{equation*}
    \tau^p\psi'(\tau^p)=p\tau^p\chi(\tau^p)^{p-1}\chi'(\tau^p)=\tau \phiop(\tau)^{p-1}\phiop'(\tau)=\phiop(\tau)^p=\chi(\tau^p)^p=\psi(\tau^p).
  \end{equation*}

  Considering the functional equation $Y(Z)=Z\psi(Y(Z))$, we see that all
  assumptions of \cite[Theorem~VI.6]{Flajolet-Sedgewick:ta:analy} are
  satisfied; thus it has a unique solution $Y(Z)$ with $Y(0)=0$ which is
  analytic around the origin. By the same result, $Y(Z)$ has radius of
  convergence
  \begin{equation*}
    \frac{\tau^p}{\psi(\tau^p)}=\frac{\tau^p}{\chi(\tau^p)^p}=\myBigl(\frac{\tau}{\phiop(\tau)}\myBigr)^p=\rho^p
  \end{equation*}
  and, as $\psi$ is aperiodic, the only singularity of $Y(Z)$ with
  $\abs{Z}\le\rho^p$ is $Z=\rho^p$.

  \newcommand{\ytilde}{\widetilde y}
  We consider the function $\ytilde(z)\coloneqq z\chi(Y(z^p))$. By definition,
  it is analytic for $\abs{z}<\rho$ and its only singularities with
  $\abs{z}\le\rho$ are those $z$ with $z^p=\rho^p$, i.e., $z=\rho\zeta$ for
  $\zeta\in G(p)$. It is also clear by definition that $[z^n]\ytilde(z)=0$
  for $n\not\equiv 1\pmod p$. We have $\ytilde(0)=0$ and
  \begin{equation*}
    z\phiop(\ytilde(z))=z\chi((\ytilde(z))^p) = z\chi(z^p \chi(Y(z^p))^p) =
    z\chi(z^p \psi(Y(z^p)))=z\chi(Y(z^p))=\ytilde(z)
  \end{equation*}
  because $z^p \psi(Y(z^p))=Y(z^p)$ by definition of $Y$. This implies that
  $y=\ytilde$.
\end{proof}

While the following proposition is particularly useful in the context of the previous one,
it also holds in a slightly more general setting. It gives a detailed description of the
singular expansions for $p$-periodic power series like above.

\begin{proposition}\label{prop:singularities-roots-of-unity}
  Let $p$ be a positive integer and let $y$ be analytic with radius of convergence 
  $0 < \rho \leq \infty$, where $[z^{n}]
  y(z) = 0$ for $n\not\equiv 1\pmod p$. Assume that $y(z)$ has $p$ dominant singularites
  located at $\zeta \rho$ for $\zeta \in G(p)$, and that for some $L \geq 0$ and $z \to
  \rho$, we have the singular expansion
  \[ y(z) \stackrel{z\to\rho}{=} \sum_{j=0}^{L-1} d_{j} \myBigl(1 - \frac{z}{\rho}\myBigr)^{-\alpha_{j}} +
    O\myBigl(\myBigl(1 - \frac{z}{\rho}\myBigr)^{-\alpha_{L}}\myBigr), \]
  where $\alpha_{0}$, $\alpha_{1}$, \dots, $\alpha_{L}$ are complex numbers such that
  $\Re(\alpha_{j}) \geq \Re(\alpha_{j+1})$ for all $0\leq j < L$.

  Then, for $\zeta\in G(p)$, the singular expansion of $y(z)$ for $z\to \zeta\rho$ is
  given by
  \[ y(z) \stackrel{z\to\zeta\rho}{=} \sum_{j=0}^{L-1} \zeta d_{j} \myBigl(1 - \frac{z}{\zeta\rho}\myBigr)^{-\alpha_{j}}
    + O\myBigl(\myBigl(1 - \frac{z}{\zeta\rho}\myBigr)^{-\alpha_{L}}\myBigr), \]
  i.e., the expansion for $z \to \zeta\rho$ can be obtained by multiplying the expansion
  for $z\to \rho$ with $\zeta$ and substituting $z\mapsto \zeta/\rho$. Finally, for the
  coefficients of $y(z)$ we find
  \begin{equation}\label{eq:p-periodic-coefficients}
    [z^{n}] y(z) = \iverson{p \mid 1-n} [z^{n}] \myBigl(p \sum_{j=0}^{L-1} d_{j}
    \myBigl(1 - \frac{z}{\rho}\myBigr)^{-\alpha_{j}} + O\myBigl(\myBigl(1 -
    \frac{z}{\rho}\myBigr)^{-\alpha_{L}}\myBigr) \myBigr),
  \end{equation}
  which can be made explicit easily by means of singularity analysis (cf.~\cite[Chapter
  VI.4]{Flajolet-Sedgewick:ta:analy}). In particular,
  \[ [z^{n}] y(z) = \iverson{p \mid 1-n}\myBigl(\sum_{j=0}^{L-1} \frac{p
      d_{j}}{\Gamma(\alpha_{j})} n^{\alpha_{j} - 1} \rho^{-n} + O\big( n^{\Re(\alpha_{0}) -
      2} \rho^{-n} \big) + O\big(n^{\Re(\alpha_L) - 1} \rho^{-n}\big)\myBigr). \]
\end{proposition}
\begin{proof}
  As $[z^{n}] y(z) = 0$ for $n \not\equiv 1 \pmod{p}$ there is a function $\chi$, analytic
  around the origin, such that $y(z) = z \chi(z^{p})$. Thus,
  for every $\zeta\in G(p)$, we have
  \begin{equation*}
    y(\zeta z) = \zeta z\chi((\zeta z)^{p}) = \zeta z \chi(z^{p}) = \zeta y(z)
  \end{equation*}
  or, equivalently,
  \begin{equation*}
    y(z)=\zeta y\myBigl(\frac{z}{\zeta}\myBigr).
  \end{equation*}
  Thus the singular expansion for $z\to\zeta\rho$ follows from that for
  $z\to\rho$ by replacing $z$ with $z/\zeta$ and multiplication by $\zeta$.

  With the singular expansions at all the dominant singularities located at $\zeta\rho$
  for $\zeta\in G(p)$ at hand, we are able to extract the overall growth of the
  coefficients of $y(z)$ by first applying singularity analysis to every expansion
  separately, and then summing up all these contributions. When doing so, we use the
  well-known property of roots of unity that
  \begin{equation}\label{eq:sum-of-zetas}
    \sum_{\zeta\in G(p)} \zeta^{m} = p \iverson{p\mid m}
  \end{equation}
  for $m\in\mathbb{Z}$ in order to rewrite the occurring sums as 
  $\sum_{\zeta\in G(p)} \zeta^{1-n} = p \iverson{p\mid 1-n}$. Comparing the resulting
  asymptotic expansion with~\eqref{eq:p-periodic-coefficients} proves the statement.
\end{proof}

The following result is a consequence of Propositions~\ref{prop:lukasiewicz-ascents}
(resp.\ Corollary~\ref{cor:lukasiewicz-gf-implicit}),
\ref{prop:singular-inversion-expansion}, and \ref{prop:singularities-roots-of-unity}. It
shows that actually we have more than enough information to carry out the asymptotic
analysis of the number of ascents, although in general we do not know the function $V(z,t)$
explicitly.

\begin{corollary}\label{cor:V-derivatives}
  Let $V(z,t)$ be the bivariate generating function from
  Corollary~\ref{cor:lukasiewicz-gf-implicit} and let $V(z) = V(z,1)$.
  \begin{enumerate}
  \item Let $\tau > 0$ be the uniquely determined positive constant satisfying
    $\PathGF'(\tau) = 0$. Then $V(z)$ has radius of convergence $\rho := 1/\PathGF(\tau)$ with a
    square-root singularity for $z\to \rho$. If $\mathcal{S}$ has period $p$, then the
    dominant singularities (i.e., singularities with modulus $\rho$) are located at
    $\zeta \rho$ with $\zeta \in G(p)$. The corresponding
    expansions are given by
    \begin{equation}\label{eq:v-expansion}
      V(z) \stackrel{z\to\zeta\rho}{=} \zeta\tau - \zeta\sqrt{\frac{2 \PathGF(\tau)}{\PathGF''(\tau)}}\, \myBigl(1
      - \frac{z}{\zeta\rho}\myBigr)^{1/2} - \zeta\frac{\PathGF(\tau) \PathGF'''(\tau)}{3 \PathGF''(\tau)^{2}}\, \myBigl(1 -
      \frac{z}{\zeta\rho}\myBigr) + O\myBigl(\myBigl(1 - \frac{z}{\zeta\rho}\myBigr)^{3/2}\myBigr).
    \end{equation}
  \item The evaluation of the partial derivatives $\frac{\partial^{\nu}}{\partial t^{\nu}}
    V(z,t)$ at $t = 1$ can be expressed in terms of~$V(z)$. For instance, the first
    partial derivative is given as
    \begin{equation}\label{eq:v-diff-1}
      V_{t}(z) = - z \frac{(V(z) - z)^{r}}{V(z)^{r+2} \PathGF'(V(z))}.
    \end{equation}
  \end{enumerate}
\end{corollary}
\begin{proof} Let $F(z,t,v)$ be given as in the statement of
  Proposition~\ref{prop:lukasiewicz-ascents}.
  \begin{enumerate}
  \item The singular expansion of $V(z)$ for $z\to \zeta \rho$ follows from
    applying Propositions~\ref{prop:singular-inversion-expansion}
    and~\ref{prop:singularities-roots-of-unity} to our given context: Plugging in $t=1$
    in~\eqref{eq:lukasiewicz-v-equation} (or, alternatively, the combinatorial
    interpretation of $V(z,t)$ as stated in Corollary~\ref{cor:lukasiewicz-gf-implicit})
    proves that $V(z)$ satisfies the
    functional equation $V(z) = z \phiop(V(z))$ with $\phiop(u) = u \PathGF(u)$. For this particular
    $\phiop(u)$, the fundamental constant $\tau$ is defined as the unique positive real number
    satisfying $\PathGF'(\tau) = 0$. Then, Proposition~\ref{prop:singular-inversion-expansion}
    yields the singular structure as well as the singular expansion for $z\to\rho$ after
    checking that $\phiop(u)$ satisfies the necessary conditions---and indeed, we have
    $\phiop(0) = 1 \neq 0$, and for step sets other than $\mathcal{S} = \{-1,
    0\}$, $\phiop$ is also a nonlinear function. With the computed expansion for $z\to\rho$,
    we obtain~\eqref{eq:v-expansion} from
    Proposition~\ref{prop:singularities-roots-of-unity}.

  \item As a consequence of $V(z,t)$ being a bivariate generating function
    where the coefficient of $z^{n}$ is given by a polynomial in $t$ (see
    Corollary~\ref{cor:lukasiewicz-gf-implicit}), and as we know that $V(z) = V(z,1)$ has
    radius of convergence $\rho = 1/\PathGF(\tau)$, we obtain that $V(z,t)$ is analytic in a
    small neighborhood of $(z,t) = (0,1)$. This allows us to implicitly differentiate the
    functional equation~\eqref{eq:V-combinatorial-func} with respect to $t$. Within the
    implicit derivative of this equation, the partial derivative
    $\frac{\partial }{\partial t} V(z,t)$ only occurs linearly, so that we can solve for
    it.

    Equation~\eqref{eq:v-diff-1} can now be obtained by setting $t = 1$ and using the
    relation $z \PathGF(V(z)) = 1$  (see Corollary~\ref{cor:lukasiewicz-gf-implicit}).
    Higher-order partial derivatives can be obtained by differentiating again with respect
    to $t$ before setting $t = 1$.\qedhere
  \end{enumerate}
\end{proof}

These observations allow us to employ singularity analysis (see, e.g.,
\cite[Chapter VI]{Flajolet-Sedgewick:ta:analy}) in order to carry out a precise analysis
of the number of $r$-ascents in certain families of \L ukasiewicz paths in the following
sections.

We conclude this section with a very useful observation with respect to the nature of the
structural constant $\tau$.

\begin{lemma}\label{lem:tau-1}
  Let $\mathcal{S}$ be some\footnote{Recall that we excluded $\mathcal{S} = \{-1, 0\}$ in
    the introduction.} \L ukasiewicz step set and let $\tau$ be the corresponding
  structural constant, i.e.\ the unique positive number satisfying $\PathGF'(\tau) = 0$. Then
  $\tau \leq 1$ with equality if and only if $\mathcal{S} = \{-1, 0, 1\}$ or $\mathcal{S} = \{-1, 1\}$.
\end{lemma}
\begin{proof}
  First, observe that $\PathGF'$ is a strictly increasing function. For $u \geq 1$, we have
  \[ \PathGF'(u) \geq \PathGF'(1) = -1 + \sum_{\substack{s\in\mathcal{S}\\ s \geq 0}} s \geq 0  \]
  with equality if and only if $u = 1$ and $\mathcal{S} \in \{\{-1, 0, 1\}, \{-1,
  1\}\}$. Monotonicity of $\PathGF'$ then implies the assertion of the lemma.
\end{proof}
\begin{remark}
  The number of \L ukasiewicz excursions of length $n$ is trivially bounded from above by
  $\abs{\mathcal{S}}^{n}$ which corresponds to all paths with the same step set but
  without any restrictions. Consequently, the radius of convergence $\rho$ of the
  generating function of excursions $V(z)/z$ is bounded from below by
  $\frac{1}{\abs{\mathcal{S}}}$.

  Assume that $\mathcal{S} \not\in \{\{-1, 0, 1\}, \{-1, 1\}\}$. In this case, $\tau < 1$
  and $\PathGF'(u) > 0$ for $\tau < u < 1$. This implies $\abs{\mathcal{S}} = \PathGF(1) > \PathGF(\tau) =
  1/\rho$, which means that the radius of convergence $\rho$ is strictly larger than the
  trivial bound $1/\abs{\mathcal{S}}$. In other words, for all but the two simple step
  sets $\{-1, 0, 1\}$ and $\{-1, 1\}$, the restriction to \L ukasiewicz excursions leads
  to an exponentially smaller number of admissible paths.
\end{remark}

The quantity $\PathGF'(1)$ is also referred to as the \emph{drift} of the walk (see,
e.g.,~\cite[Section 3.2]{Banderier-Flajolet:2002:lat-path}) and strongly influences the
asymptotic behavior of corresponding meanders.

\section{Analysis of Ascents}\label{sec:ascents}
\subsection{Analysis of Excursions}\label{sec:ascents:excursions}

In this section we focus on the analysis of \emph{excursions}, i.e., paths that start and
end on the horizontal axis. As mentioned in Section~\ref{sec:gf-analytic}, on the generating
function level, this corresponds to setting $v = 0$ in $F(z, t, v)$
from~\eqref{eq:lukasiewicz-ascents:p-plus}. Also note that from this point on it is quite
useful to replace $\PathGF_{+}(v) = \PathGF(v) - 1/v$ in  $F(z,t,v)$.

Recall that $E_{n,r}$ is the random variable modeling the number of $r$-ascents in a
random non-negative \L ukasiewicz excursion of length $n$ with respect to some given step
set $\mathcal{S}$.

\begin{theorem}\label{thm:ascents:excursions}
  Let $r\in\N$, $n\in \N_{0}$, and $p \geq 1$ be the period of the step set $\mathcal{S}$. Let $\tau$ be the
  structural constant, i.e., the unique positive solution of $\PathGF'(\tau) = 0$. Set $c := \tau \PathGF(\tau)$.

  Then, the expected number of $r$-ascents in \L ukasiewicz
  paths of length $n$ for $n\equiv 0\pmod{p}$ as well as the corresponding variance grow with $n\to\infty$
  according to the asymptotic expansions
  \begin{align}\begin{split}\label{eq:ascents:excursions:expectation}
    \E E_{n,r} = \frac{(c - 1)^{r}}{c^{r+2}} n
    + \frac{(c - 1)^{r-2}}{2 \tau^{2}c^{r+2} \PathGF''(\tau)^{2}}\Big(&
    \PathGF''(\tau)^{2}\tau^{2} \big(4c^{2} - (r+8)c + r+4\big)\\
    &- \PathGF''(\tau)\PathGF(\tau) \big(6c^{2} - 6(r+2) c + r^{2} + 5r +
    6\big)\\
    &- \PathGF'''(\tau)c(2c^{2} - (r+4)c + r+2)
    \Big)  + O(n^{-1/2})
    \end{split}
  \end{align}
  and
  \begin{align}\begin{split}\label{eq:ascents:excursions:variance}
      \V E_{n,r} = \myBigl(\frac{(c-1)^{r}}{c^{r+2}} + \frac{(2c - 2r -
        3)(c-1)^{2r}}{c^{2r+4}} - \frac{(c-1)^{2r-2} (2c - r - 2)^{2}}{c^{2r+3} \tau^{3}
        \PathGF''(\tau)} \myBigr) n + O(n^{1/2}).
  \end{split}\end{align}
  Additionally, for $n \not\equiv 0 \mod p$, we have $E_{n,r} = 0$. All $O$-constants
  depend implicitly on $r$.
\end{theorem}
\begin{proof}
  While the proof for this theorem actually is quite straightforward, it involves
  some rather computationally expensive operations with asymptotic expansions, which we
  have carried out with the help of SageMath, see the corresponding worksheet as
  referenced at the end of Section~\ref{sec:introduction}.

  As discussed in the remark after the definition of periodic lattice paths, for a
  $p$-periodic step set $\mathcal{S}$ there are no excursions of length $n$ for $n
  \not\equiv 0\mod p$. Thus, the random variable $E_{n,r}$ degenerates to the constant $0$
  in these cases,
  allowing us to focus on the case where $n$ is a multiple of $p$.

  Based on the fact that the generating function enumerating $r$-ascents within \L
  ukasiewicz excursions is given by $F(z,t,0) = V(z,t)/z$, our general strategy for
  determining asymptotic expansions for the expected number of
  $r$-ascents and the corresponding variance is to compute expansions for the first and
  second factorial moment by normalizing the extracted coefficients of $V_{t}(z)/z$ and
  $V_{tt}(z)/z$.

  In order to normalize these extracted coefficients, we need to compute an asymptotic
  expansion for the number of $\mathcal{S}$-excursions of given length. To accomplish
  this, we could simply use the general framework developed by Banderier and Flajolet
  in~\cite[Theorem 3]{Banderier-Flajolet:2002:lat-path}---however, we choose to analyze
  this quantity more directly by applying singularity analysis to the singular expansion
  of $V(z) =  V(z,1)$ given in~\eqref{eq:v-expansion}. With the help of SageMath and
  Proposition~\ref{prop:singularities-roots-of-unity}, we immediately find
  \begin{align}\begin{split}\label{eq:V-coefs-asy}
      [z^{n}] \frac{V(z)}{z} & = p \sqrt{\frac{\PathGF(\tau)^{3}}{2\pi \PathGF''(\tau)}}\, \PathGF(\tau)^{n}
      n^{-3/2} \\
      & \quad  - \frac{p}{24} \sqrt{\frac{\PathGF(\tau)^{3} }{2\pi \PathGF''(\tau)^{7}}}\, \big(
      45 \PathGF''(\tau)^{3} + 5 \PathGF(\tau) \PathGF'''(\tau)^{2}
       \\ & \qquad\qquad\qquad\qquad\qquad - 3 \PathGF(\tau) \PathGF''(\tau) \PathGF''''(\tau)
      \big) \PathGF(\tau)^{n} n^{-5/2} \\ & \quad + O(\PathGF(\tau)^{n} n^{-3})
    \end{split}\end{align}
  for $n \equiv 0 \pmod{p}$.

  For the analysis of the expected value, we turn our attention to $V_{t}(z)/z$. By
  Corollary~\ref{cor:V-derivatives}, the derivative $V_{t}(z)$ can be expressed in terms
  of $V(z)$. Then, as $V(z)$ and $V_{t}(z)$ are both analytic with radius of convergence
  $\rho$, we can use the singular expansion of $V(z)$ together with~\eqref{eq:v-diff-1} to
  arrive at a singular expansion of $V_{t}(z)/z$. By~\eqref{eq:V-power-series} we know
  that Proposition~\ref{prop:singularities-roots-of-unity} can be used to obtain the
  singular expansion of $V_{t}(z)$ for $z\to \zeta\rho$ with $\zeta\in G(p)$.

  This allows us to extract the $n$th coefficient of
  $V_{t}(z)/z$ by means of singularity analysis---and dividing by 
  the growth of the number
  of excursions of length $n$ yields an asymptotic expansion for $\E E_{n,r}$. This
  gives~\eqref{eq:ascents:excursions:expectation}.

  In the same manner, by investigating the second derivative $V_{tt}(z)/z$, we can obtain
  an asymptotic expansion for the second factorial moment, $\E (E_{n,r}
  (E_{n,r}-1))$. Then, applying the well-known identity
  \[ \V E_{n,r} = \E (E_{n,r}(E_{n,r}-1)) + \E E_{n,r} - (\E E_{n,r})^{2}, \]
  proves~\eqref{eq:ascents:excursions:variance}.
\end{proof}

By means of Theorem~\ref{thm:ascents:excursions} we are immediately able to determine the
asymptotic behavior of interesting special cases. We are particularly interested in the
most basic setting: $\mathcal{S} = \{-1, 1\}$, i.e., Dyck paths.

\begin{example}[$r$-Ascents in Dyck paths]\label{example:r-ascents-dyck}
  In the case of Dyck paths, we have $u \PathGF(u) = 1 + u^{2}$. From there, it is easy to see
  that $\tau = 1$ and $\rho = 1/2$, and that the family of paths is $2$-periodic. By the
  same approach as in the proof of Theorem~\ref{thm:ascents:excursions}, we
  can determine the expected number and variance of $r$-ascents in Dyck paths of length
  $2n$ with higher precision than stated in Theorem~\ref{thm:ascents:excursions}, namely as
  \[
    \E D_{2n,r} = \frac{n}{2^{r+1}} - \frac{(r+1)(r-4)}{2^{r+3}} \\ + \frac{(r^{2} - 11r +
      22)(r+1)r}{2^{r+6}} n^{-1} + O(n^{-2})
  \]
  and
  \begin{multline*}
    \V D_{2n,r} = \myBigl(\frac{1}{2^{r+1}} - \frac{r^{2} - 2r + 3}{2^{2r+3}}\myBigr)n \\
    - \myBigl(\frac{r^{2} - 3r - 4}{2^{r+3}} - \frac{3r^{4} - 20r^{3} + 29r^{2} - 10r -
      14}{2^{2r+5}}\myBigr) + O(n^{-1/2}).
  \end{multline*}
  However, as we have a closed expression for $V(z)$, we can do even better. Because of
  \[ \frac{V(z)}{z} = \frac{1 - \sqrt{1 - 4z^{2}\,}}{2z^{2}}, \]
  we can also write down the generating function $V_{t}(z)/z$ for the expected number of
  $r$-ascents explicitly. We find the $2$-periodic power series
  \begin{equation}\label{eq:ascents:dyck:gf-expectation}
    \frac{V_{t}(z)}{z} = \frac{(1 - \sqrt{1 - 4z^{2}\,})^{r} (1 + \sqrt{1 - 4z^{2}})}{2^{r+1} \sqrt{1 - 4z^{2}\,}},
  \end{equation}
  for which, after substituting $Z = z^{2}$, we can apply Cauchy's integral formula in
  order to extract the expected values before normalization with the $n$th Catalan number
  $C_{n}$ explicitly.

  Considering a contour $\gamma$ that stays sufficiently close to the origin and winds
  around it exactly once, and by the integral substitution $Z = \frac{u}{(1+u)^{2}}$ we
  obtain
  \begin{align*}
    [Z^{n}] \frac{V_{t}(z)}{z} & = \frac{1}{2\pi i} \oint_{\gamma} \frac{V_{t}(z)/z}{Z^{n+1}}~dZ\\
                & = \frac{1}{2\pi i} \oint_{\tilde\gamma} \myBigl(\frac{u}{1+u}\myBigr)^{r}
                  \frac{1}{1 - u} \frac{(1 + u)^{2n+2}}{u^{n+1}} \frac{1 - u}{(1 +
                  u)^{3}}~du\\
                & = \frac{1}{2\pi i} \oint_{\tilde\gamma} \frac{(1 +
                  u)^{2n-r-1}}{u^{n-r+1}}~du = [u^{n-r}] (1+u)^{2n-r+1}\\
                & = \binom{2n-r-1}{n-1},
  \end{align*}
  where $\tilde\gamma$ denotes the image of $\gamma$ under the transformation; a curve
  that also stays close to the origin and winds around it once. After normalization, this
  proves the exact formula
  \[ \E D_{2n,r} = \frac{1}{C_{n}} \binom{2n-r-1}{n-1}. \]
\end{example}

\subsection{Analysis of Dispersed Excursions}\label{sec:ascents:dispersed}
Let $\mathcal{S}$ be a \L ukasiewicz step set where $0 \not\in\mathcal{S}$. In this
setting, we define a \emph{dispersed \L ukasiewicz excursion} to be an
$\mathcal{S}$-excursion where, additionally, horizontal steps ``$\rightarrow$'' can be
taken whenever the path is on its starting altitude. Observe that, by our definition of $r$-ascents, these
horizontal steps do not contribute towards ascents, as only the non-negative steps from
$\mathcal{S}$ are relevant.

The motivation to study this specific family of \L ukasiewicz paths originates
from~\cite{Kangro-Pourmoradnasseri-Theis:2016:ascents-dispersed-dyck}, where the authors
investigate the total number of $1$-ascents in dispersed Dyck paths using elementary
methods. Our goal in this section is to find asymptotic expansions for the number of
dispersed \L ukasiewicz excursions of given length as well as for the expected number of
$r$-ascents in these paths.

We begin our analysis by constructing a suitable bivariate generating function
enumerating dispersed \L ukasiewicz excursions with respect to their length and the number
of $r$-ascents.

\begin{proposition}
  Let $r\in\N$ and $V(z,t)$ as in Proposition~\ref{prop:lukasiewicz-ascents} or
  Corollary~\ref{cor:lukasiewicz-gf-implicit}. Then the
  generating function $D(z,t)$ enumerating dispersed $\mathcal{S}$-excursions where $z$
  marks the length of the excursion and $t$ marks the number of $r$-ascents is given by
  \begin{equation}\label{eq:dispersed:gf}
    D(z,t) = \frac{1}{z} \frac{V(z,t)}{1 - V(z,t)}.
  \end{equation}
\end{proposition}
\begin{proof}
  Let $\mathcal{E}$ denote the combinatorial class of $\mathcal{S}$-excursions. The
  corresponding bivariate generating function is given by $V(z,t)/z$, as proved in
  Corollary~\ref{cor:lukasiewicz-gf-implicit}.

  By the symbolic method (see~\cite[Chapter I]{Flajolet-Sedgewick:ta:analy}), the combinatorial class
  $\mathcal{D}$ of dispersed excursions can be constructed as
  \[ \mathcal{D} = (\mathcal{E} \rightarrow)^{*} \mathcal{E}.  \]
  Translating this combinatorial construction in the language of
  (bivariate) generating functions, we find
  \[ D(z,t) = \frac{1}{1 - \frac{V(z,t)}{z} z} \frac{V(z,t)}{z}, \]
  and simplification immediately yields~\eqref{eq:dispersed:gf}.
\end{proof}

In preparation for the analysis of the generating function $D(z,t)$, we have to
investigate the structure of the dominant singularities. In particular, the following
lemma states that in most cases, the dominant singularity of $D(z,1)$ comes from the
dominant square root singularities on the radius of convergence of $V(z)$.

\begin{lemma}\label{lem:v-is-less-than-1}
  The radius of convergence of $D(z,1)$ (as well as for the corresponding partial
  derivatives with respect to $t$, i.e., $\frac{\partial^{\nu}}{\partial t^{\nu}}
  D(z,t)|_{t=1}$) is given by $\rho = 1/\PathGF(\tau)$, where $\tau > 0$ is the structural
  constant with respect to the step set $\mathcal{S}$.
\end{lemma}
\begin{proof}
  As $V(z)$ is a power series with non-negative coefficients, we have
  \[ \abs{V(z)} \leq V(\abs{z}) \leq V(\rho) = \tau \leq 1  \]
  for $\abs{z}\leq \rho$ by Lemma~\ref{lem:tau-1}. By the same lemma and because we
  assumed $0\not\in\mathcal{S}$, equality holds only
  in case of $\mathcal{S} = \{-1, 1\}$. Thus, the denominator $1 -
  V(z)$ of $D(z,1)$ does not contribute a pole for $\abs{z} < \rho$.
\end{proof}

Lemma~\ref{lem:v-is-less-than-1} tells us that in the general case of $\tau \neq 1$, 
the singularities of $D(z,1)$ are of the same type as the singularities of $V(z)$. Therefore,
the precise description of the singular structure of $V(z)$ given in 
Corollary~\ref{cor:V-derivatives} allows us to carry out the asymptotic analysis.

Recall that $D_{n,r}$ is the random variable modeling the number of $r$-ascents in a
random dispersed \L ukasiewicz excursion of length $n$ with respect to some step set
$\mathcal{S}$.

\begin{theorem}\label{thm:dispersed:tau-not-1}
  Let $p \geq 1$ be the period of the step set $\mathcal{S}$. Assume additionally that for
  the structural constant $\tau$ we have $\tau \neq 1$.

  Then $d_{n}$, the number of dispersed \L ukasiewicz excursions of length $n$, satisfies
  \begin{align}\begin{split}\label{eq:dispersed:tau-not-1:number}
      d_{n} & = \frac{1}{\sqrt{2\pi}} \frac{p\tau^{k} (\tau^{p}(p-k-1) + k+1)}{(1 -
        \tau^{p})^{2}} \sqrt{\frac{\PathGF(\tau)^{3}}{\PathGF''(\tau)}} \PathGF(\tau)^{n}n^{-3/2} + O(\PathGF(\tau)^{n} n^{-5/2})
    \end{split}\end{align}
  for $n\equiv k\mod p$ and $0\leq k\leq p-1$.
  Furthermore, the expected number of $r$-ascents grows with $n\to\infty$ according to the
  asymptotic expansion
  \begin{equation}\label{eq:dispersed:tau-not-1:expectation}
    \E D_{n,r} = \frac{(\tau \PathGF(\tau) - 1)^{r}}{(\tau \PathGF(\tau))^{r+2}} n + O(1).
  \end{equation}
  The $O$-constants depend implicitly on both $r$ as well as on the residue class
  of $n$ modulo $p$.
\end{theorem}

  In a nutshell, the proof of this theorem involves a rigorous analysis of the generating functions
  $D(z,1)$ (for the overall number of dispersed excursions), as well as of $D_{t}(z,1) =
  \frac{1}{z} \frac{V_{t}(z)}{(1 - V(z))^{2}}$ (for the expected number of ascents in these
  paths). All computations are carried out in the corresponding SageMath worksheet,
  \href{https://benjamin-hackl.at/downloads/lukasiewicz-ascents/lukasiewicz-dispersed-excursions.ipynb}{\texttt{lukasiewicz-dispersed-excursions.ipynb}}. Furthermore,
  while our results as stated in~\eqref{eq:dispersed:tau-not-1:number}
  and~\eqref{eq:dispersed:tau-not-1:expectation} only list the asymptotic main term, expansions
  with higher precision are available in the worksheet as well (they just become rather
  messy very quickly).
  
\begin{proof}
  Before we delve into the analysis, let us recall the setting we have to deal with. As
  the period of the step set $\mathcal{S}$ is $p$, the function $V(z)$ (and
  corresponding derivatives with respect to $t$) has $p$ dominant square root
  singularities, located at $\zeta/\PathGF(\tau)$ with $\zeta \in G(p)$ (see
  Corollary~\ref{cor:V-derivatives}).

  Furthermore, as $\tau \neq 1$, Lemma~\ref{lem:v-is-less-than-1} tells us that the
  singular structure of $D(z,1)$ (and its derivatives with respect to $t$) is directly
  inherited from $V(z)$, meaning that the singularities of $D(z,1)$ are of the same type
  as the singularities of $V(z)$.

  Thus, after rewriting
  \[ D(z,1) = \frac{1}{z}\frac{V(z)}{1 - V(z)} = \frac{1}{z} (
    \frac{1}{1 - V(z)} - 1), \]
  we can use the expansion of $V(z)$ for $z\to
  \zeta/\PathGF(\tau)$ from~\eqref{eq:v-expansion} to compute the expansion for $D(z,1)$
  with $z\to \zeta/\PathGF(\tau)$, yielding
  \[
    D(z,1) \stackrel{z\to \zeta/\PathGF(\tau)}{=} \frac{\tau \PathGF(\tau)}{1 - \tau \zeta} -
    \frac{1}{(1 - \tau\zeta)^{2}} \sqrt{\frac{2 \PathGF(\tau)^{3}}{\PathGF''(\tau)}} \big(1 -
    \frac{z}{\zeta/\PathGF(\tau)}\big)^{1/2} + O\big(1 -
    \frac{z}{\zeta/\PathGF(\tau)}\big).
  \]
  By applying singularity analysis we are able to determine the contribution of the
  singularity located at $\zeta/\PathGF(\tau)$ to the overall growth of the coefficients of
  $D(z,1)$, which can then be obtained by summing up the contributions of all
  singularities on the radius of convergence. In our case, this translates to summing over
  all $p$th roots of unity $\zeta \in G(p)$.

  After doing so, we see that in the main term all roots of unity can be grouped together
  such that we find
  \[ \frac{1}{\sqrt{2\pi}} \myBigl(\sum_{\zeta\in G(p)} \frac{\zeta^{-n}}{(1 - \tau
      \zeta)^{2}}\myBigr) \sqrt{\frac{\PathGF(\tau)^{3}}{\PathGF''(\tau)}} \PathGF(\tau)^{n}n^{-3/2}.  \]
  In fact, when studying these expansions with higher precision, the corresponding sums
  that occur have the shape
  \[ \sum_{\zeta\in G(p)} \frac{\zeta^{\ell - n}}{(1 - \tau \zeta)^{m}}  \]
  for some integers $\ell$, $m \geq 0$. To find an explicit expression for this sum, we
  first recall~\eqref{eq:sum-of-zetas} as well as another elementary property of roots of unity, namely
  \[ (1 - \zeta\tau) \sum_{j=0}^{p-1} (\tau\zeta)^{j} = 1 - \tau^{p}. \]
  Now let $n \equiv k\mod p$ with $0\leq k \leq p-1$. Then we can rewrite
  \[ \sum_{\zeta\in G(p)} \frac{\zeta^{\ell - n}}{(1 - \tau\zeta)^{m}} = \frac{1}{(1 -
      \tau^{p})^{m}} \sum_{\zeta\in G(p)} \zeta^{\ell - k} (1 + \tau\zeta +
    (\tau\zeta)^{2} + \dots + (\tau\zeta)^{p-1})^{m}.  \]
  By~\eqref{eq:sum-of-zetas} we only need to determine those summands
  in $(1 + \tau \zeta + \dots + (\tau\zeta)^{p-1})^{m}$ involving
  $\tau^{j}$ with $0\leq j\leq m(p-1)$ and $j \equiv k-\ell\mod p$. This can
  be done easily for explicitly given values of $m$ and $\ell$, for example
  \[ \sum_{\zeta\in G(p)} \frac{\zeta^{-n}}{(1 - \tau\zeta)^{2}} = \frac{p \tau^{k}
      (\tau^{p}(p-k-1) + k+1)}{(1 - \tau^{p})^{2}}.  \]
  Plugging this into the previously obtained asymptotic expansion for the growth of the coefficients of $D(z,1)$
  yields~\eqref{eq:dispersed:tau-not-1:number}.

  For the expected value we focus on the generating function
  \[ D_{t}(z,1) = \frac{V_{t}(z)}{z} \frac{1}{(1 - V(z))^{2}} \]
  and proceed similarly to
  above. Using~\eqref{eq:v-expansion} and~\eqref{eq:v-diff-1} we are again able to compute
  the singular expansion of $D_{t}(z,1)$ for $z\to \zeta/\PathGF(\tau)$, namely
  \[
    \frac{1}{(1 - \tau\zeta)^{2}} \frac{(\tau \PathGF(\tau) - 1)^{r}}{\PathGF(\tau)^{r} \tau^{r+2}
      \sqrt{2 \PathGF(\tau) \PathGF''(\tau)}} \big(1 - \frac{z}{\zeta/\PathGF(\tau)}\big)^{-1/2} + O(1).
  \]
  Extracting the contributions of the singularity at $\zeta/\PathGF(\tau)$, summing up the
  contributions of all $p$ singularities, and then finally normalizing the result by
  dividing by the overall number of dispersed excursions of length $n$ we arrive
  at~\eqref{eq:dispersed:tau-not-1:expectation}.
\end{proof}

By Lemma~\ref{lem:tau-1}, the only family of \L ukasiewicz paths that is not covered by
Theorem~\ref{thm:dispersed:tau-not-1} is $\mathcal{S} = \{-1, 1\}$, the case of dispersed
Dyck paths. However, as everything is explicitly given, the analysis is quite straightforward.

\begin{proposition}\label{prop:ascents-dispersed-dyck}
  Let $d_{n}$ denote the total number of dispersed Dyck paths of length $n$, and let
  $D_{n,r}$ denote the random variable modeling the number of $r$-ascents in a random
  dispersed Dyck path of length $n$.

  Then, $d_{n}$ is given by
  \begin{equation}\label{eq:dyck-dispersed-number}
    d_{n} = \binom{n}{\lfloor n/2 \rfloor} = \sqrt{\frac{2}{\pi}} 2^{n} n^{-1/2} - \frac{2
      - (-1)^{n}}{2\sqrt{2\pi}} 2^{n} n^{-3/2} + O(2^{n} n^{-5/2}),
  \end{equation}
  and the expected number of $r$-ascents satisfies
  \begin{align}
    \begin{split}\label{eq:dispersed:expectation:dyck}
      \E D_{n,r} & = \frac{n}{2^{r+2}} - \sqrt{\frac{\pi}{2}} \frac{r-2}{2^{r+2}} n^{1/2} +
      \frac{(r-1)(r-4)}{2^{r+3}} \\ & \qquad  - \sqrt{\frac{\pi}{2}} \frac{(r-2)(2 - (-1)^{n})}{2^{r+4}}
      n^{-1/2} + O(n^{-1}).
    \end{split}
  \end{align}
\end{proposition}
\begin{proof}
  In the case where $\tau = 1$, the zero in the denominator of $\frac{1}{1 - V(z)}$ combines with the
  square root singularity from $V(z)$ itself. From Example~\ref{example:r-ascents-dyck} we
  know that
  \[ V(z) = \frac{1 - \sqrt{1 - 4z^{2}}}{2z}\quad\text{ and }\quad V_{t}(z) = z \frac{(1 -
      \sqrt{1 - 4z^{2}}\,)^{r} (1 + \sqrt{1 - 4z^{2}}\,)}{2^{r+1} \sqrt{1 - 4z^{2}}}.  \]
  The number $d_{n}$ of dispersed Dyck paths of length $n$ can be read off as the
  coefficients of
  \[ D(z,1) = \frac{1}{z} \frac{V(z)}{1 - V(z)} = \frac{1}{2z} \myBigl(\sqrt{\frac{1 + 2z}{1
        - 2z}} - 1\myBigr)  = \frac{1}{\sqrt{1 - 4z^{2}}} + \frac{1}{2z}
    \myBigl(\frac{1}{\sqrt{1 - 4z^{2}}} - 1\myBigr). \]
  This proves~\eqref{eq:dyck-dispersed-number}, where the asymptotic part can be obtained
  by means of singularity analysis. The explicit formula for $d_{n}$
  in~\eqref{eq:dyck-dispersed-number} is also stated in~\cite[Lemma
  2]{Kangro-Pourmoradnasseri-Theis:2016:ascents-dispersed-dyck}.

  For the expected number of $r$-ascents, we consider
  \[ D_{t}(z,1) = \frac{1}{z} \frac{V_{t}(z)}{(1 - V(z))^{2}} = \frac{(1 - \sqrt{1 - 4z^{2}})^{r} (1 -
      4z^{2} + \sqrt{1 - 4z^{2}} (1 - 2z^{2}))}{2^{r+1} (1 + 2z)(1 - 2z)^{2}}.  \]
  Just as before, the coefficients of this function can also be extracted by means of
  singularity analysis; the dominant singularities can be found at $z = \pm
  1/2$. Extracting the coefficients and dividing by $d_{n}$
  yields~\eqref{eq:dispersed:expectation:dyck}.
\end{proof}

This completes our analysis of $r$-ascents in dispersed \L ukasiewicz excursions.

\subsection{Analysis of Meanders}\label{sec:ascents:meanders}

In this section we study ascents in meanders, i.e., non-negative \L ukasiewicz paths
without further restriction. The corresponding generating function can be obtained
from~\eqref{eq:lukasiewicz-ascents:p-plus} by setting $v = 1$, which allows arbitrary ending
altitude of the path.

In accordance to the results from~\cite[Theorem 4]{Banderier-Flajolet:2002:lat-path}, the
behavior of meanders depends on the sign of the drift (i.e., the quantity $\PathGF'(1)$). The
following theorem handles the case of positive drift (which, in our setting, is equivalent
to $\tau \neq 1$).

Recall that $M_{n,r}$ is the random variable modeling the number of $r$-ascents in a
random non-negative \L ukasiewicz path of length $n$ with respect to some given step set
$\mathcal{S}$.

\begin{theorem}\label{thm:meanders:tau-not-1}
  Let $\tau > 0$ be the structural constant, i.e., the unique positive
  solution of $\PathGF'(\tau) = 0$, and assume that $\tau \neq 1$.

  Then, with $\xi = 1/\PathGF(1)$, the expected number of $r$-ascents in \L ukasiewicz meanders
  of length $n$ as well
  as the corresponding variance grow with $n\to\infty$ according to the asymptotic
  expansions
  \begin{multline}\label{eq:meanders-expectation}
      \E M_{n,r} = \mu n + \frac{(\PathGF(1) - 1)^{r} (2
        \PathGF(1) - 1 - r)}{\PathGF(1)^{r+2}} + \frac{(\PathGF(1) - 1)^{r} V_{z}(\xi)}{\PathGF(1)^{r+1}(1 -
        V(\xi))} - \frac{V_{t}(\xi)}{1 - V(\xi)}
      \\ + O\myBigl(n^{5/2} \myBigl(\frac{\PathGF(\tau)}{\PathGF(1)}\myBigr)^{n}\myBigr),
  \end{multline}
  and
  \begin{equation}\label{eq:meanders-variance}
    \V M_{n,r} = \sigma^2 n + O(1),
  \end{equation}
  where $\mu$ and $\sigma^2$ are given by
  \[ \mu = \frac{(\PathGF(1) - 1)^{r}}{\PathGF(1)^{r+2}} \quad \text{ and }\quad
    \sigma^2 = \frac{(\PathGF(1)-1)^{r}}{\PathGF(1)^{r+2}} + \frac{(\PathGF(1)-1)^{2r}(2\PathGF(1) - 3 - 2r)}{\PathGF(1)^{2r+4}}.
  \]
  Moreover, for $n\to\infty$, $M_{n,r}$ is asymptotically normally distributed, i.e., for
  $x\in \R$ we have
  \[
    \P\myBigl(\frac{M_{n,r} - \mu n}{\sqrt{\sigma^{2} n}} \leq x \myBigr) =
    \frac{1}{\sqrt{2\pi}} \int_{-\infty}^{x} e^{-t^{2}/2}~dt + O(n^{-1/2}).
  \]
  All $O$-constants depend implicitly on $r$.
\end{theorem}
\begin{proof}
  Just as in the analysis of excursions and dispersed excursions, the first quantity we
  require is $m_{n}$, the total number of meanders of length $n$ associated to
  $\mathcal{S}$. Setting $v = t = 1$ in~\eqref{eq:lukasiewicz-ascents:p-plus} and simplification yields
  \[ F(z,1,1) = \frac{1 - V(z)}{1 - z\PathGF(1)}.  \]
  From Corollary~\ref{cor:V-derivatives} we know that $V(z)$ has radius of
  convergence $1/\PathGF(\tau)$. As $\PathGF(u)$ is strictly convex for $u > 0$ and $\tau$ solves
  $\PathGF'(\tau) = 0$, this means that $\PathGF(\tau) < \PathGF(u)$ for all $u > 0$ with $u \neq
  \tau$. Hence, as $1/\PathGF(1) < 1/\PathGF(\tau)$, the dominant singularity of $F(z,1,1)$ is
  the simple pole located at $\xi := 1/\PathGF(1)$. Extracting coefficients yields
  \begin{equation}\label{eq:meanders-count}
    m_{n} = [z^{n}]F(z,1,1) = (1 - V(\xi)) \PathGF(1)^{n} + O(n^{-3/2} \PathGF(\tau)^{n}).
  \end{equation}
  The error term can directly be deduced from the fact that the next relevant singularity
  of $F(z,1,1)$ is of square-root type with modulus $1/\PathGF(\tau)$ (coming from
  $V(z)$). Observe that~\eqref{eq:meanders-count} could also have been obtained by
  applying~\cite[Theorem 4]{Banderier-Flajolet:2002:lat-path} for our given step set
  $\mathcal{S}$.

  In order to determine the expectation $\E M_{n,r}$, we differentiate $F(z,t,1)$ with
  respect to $t$, set $t = 1$, and then extract the coefficients of the resulting
  generating function. By construction, these coefficients are $m_{n}\cdot \E M_{n,r}$,
  allowing us to obtain an asymptotic expansion after normalizing the result.

  Carrying out the computations leads to
  \[ \frac{\partial }{\partial t} F(z,t,1)\big|_{t=1} = \frac{(cz)^{r}(cz - 1)^{2} (1
      - V(z))}{(1 - z\PathGF(1))^{2}} - \frac{V_{t}(z)}{1 - z\PathGF(1)},  \]
  where, for the sake of brevity, we define $c := \PathGF(1) - 1 = \PathGF_{+}(1)$. Determining the
  growth of the coefficients then gives
  \begin{multline*}
    \frac{c^{r}(1 - V(\xi))}{\PathGF(1)^{r+2}} n \PathGF(1)^{n} + \myBigl(\frac{(1 - V(\xi)) c^{r} (2
      \PathGF(1) - 1 - r) + c^{r} \PathGF(1) V_{z}(\xi)}{\PathGF(1)^{r+2}} - V_{t}(\xi)\myBigr) \PathGF(1)^{n} \\+
    O(n^{-1/2} \PathGF(\tau)^{n}),
    \end{multline*}
  which, after dividing by the expansion of $m_{n}$,
  proves~\eqref{eq:meanders-expectation}.

  Analogously, for the second partial derivative of $F(z,t,1)$ with respect to $t$, we find
  \[ \frac{\partial^{2} }{\partial t^{2}} F(z,t,1)\big|_{t=1} = -\frac{2z (cz)^{2r}
      (cz - 1)^{3} (1 - V(z))}{(1 - z\PathGF(1))^{3}} - \frac{2 (cz)^{r} (cz - 1)^{2}
      V_{t}(z)}{(1 - z\PathGF(1))^{2}} - \frac{V_{tt}(z)}{1 - z\PathGF(1)},  \]
  allowing us to determine the asymptotic growth of the unnormalized second factorial
  moment, $m_{n} \E (M_{n,r} (M_{n,r} - 1))$. Dividing by $m_{n}$ and computing the
  variance by means of $\V M_{n,r} = \E(M_{n,r} (M_{n,r}-1)) + \E M_{n,r} - (\E
  M_{n,r})^{2}$ then yields~\eqref{eq:meanders-variance}.

  In order to prove that $M_{n,r}$ is asymptotically normally distributed, we observe that
  \[ F(z,t,1) = \frac{(1 - V(z,t))(1 + (t-1)(1 - cz)(cz)^{r})}{1 - z\PathGF(1) - (t-1)(1 -
      cz)(cz)^{r}} \]
  has a unique simple pole at $z = 1/\PathGF(1)$ for $t = 1$ and, by Rouch\'e's theorem for
  $\abs{z} < \rho$, it has a single pole for sufficiently small $\abs{t-1}$. This allows
  us to apply the theorem on singularity pertubation for meromorphic
  functions~\cite[Theorem IX.9]{Flajolet-Sedgewick:ta:analy}, which proves the normal
  limiting distribution and thus concludes this proof.
\end{proof}
\begin{remark}[Computation of constants]
  Within~\eqref{eq:meanders-expectation} and~\eqref{eq:meanders-variance}, the asymptotic
  expansions for the expected number and variance of $r$-ascents in meanders, constants of
  the type $V(\xi)$, $V_{z}(\xi)$, $V_{t}(\xi)$, where $\xi = 1/\PathGF(1)$, occur. For higher
  precision than in the theorem, also higher derivatives (as well as mixed derivatives) occur.

  Although the function $V(z,t)$ is only given implicitly, by the following observation
  all of those constants can actually be computed.
  By taking the functional equation~\eqref{eq:lukasiewicz-gf-implicit} and rewriting it as
  \[ \frac{1}{z} = \PathGF(V(z)),  \]
  we can see that for $z = 1/\PathGF(1)$ we obtain the relation
  \[ \PathGF(1) = \PathGF\myBigl(V\myBigl(\frac{1}{\PathGF(1)}\myBigr)\myBigr). \]
  Then,
  because we know that $V(1/\PathGF(1)) > 0$ and that $\PathGF(u)$ is strictly convex for $u > 0$, the
  constant can be determined as the unique positive solution of $\PathGF(u) = \PathGF(1)$ satisfying $u
  \neq 1$.

  For determining the value of the constants involving derivatives, we make use of the
  same approach as used in Corollary~\ref{cor:V-derivatives}. By means of implicit
  differentiation we are able to rewrite any derivative of the form
  $\frac{\partial^{\nu_{1}+\nu_{2}}}{\partial t^{\nu_{1}} \partial z^{\nu_{2}}} V(z,t)|_{t=1}$ in terms of $V(z)$,
  allowing us to express all constants in terms of $V(\xi)$.
\end{remark}

By Lemma~\ref{lem:tau-1}, Theorem~\ref{thm:meanders:tau-not-1} covers all step sets except
for $\mathcal{S} = \{-1, 1\}$ and $\mathcal{S} = \{-1, 0, 1\}$. In these cases, we
have a similar situation to what we had in Section~\ref{sec:ascents:dispersed}: the square
root singularity coming from $V(z)$ combines with the zero in the denominator.

The following propositions close this gap.
\begin{proposition}\label{prop:meanders:dyck}
  The expected number of
  $r$-ascents in the \L ukasiewicz meanders of length $n$ associated to $\mathcal{S} =
  \{-1, 1\}$ as
  well as the corresponding variance grow with $n\to\infty$ according to the asymptotic
  expansions
  \begin{equation}\label{eq:meanders-expectation:dyck}
    \E M_{n,r} = \frac{n}{2^{r+2}} + \frac{\sqrt{2\pi} (r-2)}{2^{r+3}} n^{1/2} -
    \frac{r^{2} - r - 8}{2^{r+3}} + \frac{\sqrt{2\pi} ((2 - (-1)^{n})(r-2)}{2^{r+5}} n^{-1/2} + O(n^{-1}),
  \end{equation}
  and
  \begin{multline}\label{eq:meanders-variance:dyck}
    \V M_{n,r} = \frac{2^{r+3} - r^{2}(\pi - 2) + 4r (\pi - 3) - 4\pi + 10}{2^{2r+5}} n \\
    + \frac{\sqrt{2\pi} (2^{r+2} (r-2) - r^{3} + 3r^{2} - 2r + 4)}{2^{2r+5}} n^{1/2} + O(1).
  \end{multline}
\end{proposition}
\begin{proof}
  The analysis of ascents in this case is pretty much straightforward, as\footnote{See also Example~\ref{example:r-ascents-dyck}.}
  \[ V(z) = \frac{1 - \sqrt{1 - 4z^{2}}}{2z}, \]
  and
  therefore all generating functions involved in this analysis are given explicitly.

  The total number of meanders can now either be obtained by following~\cite[Theorem
  4]{Banderier-Flajolet:2002:lat-path}, or simply by extracting the coefficients of
  \[ F(z,1,1) = - \frac{1 - 2z - \sqrt{1 - 4z^{2}}}{2z (1 - 2z)},  \]
  where the dominant singularities are located at $z = \pm 1/2$. In any case, we find that
  for $n \to \infty$ there are
  \begin{equation}\label{eq:meanders:dyck-number-of}
    \sqrt{\frac{2}{\pi}} 2^{n} n^{-1/2} - \sqrt{\frac{2}{\pi}} \frac{2 - (-1)^{n}}{4}
    2^{n} n^{-3/2} + \sqrt{\frac{2}{\pi}} \frac{13 - 12 (-1)^{n}}{32} 2^{n} n^{-5/2} +
    O(2^{n} n^{-7/2})
  \end{equation}
  meanders with steps $\mathcal{S} = \{-1, 1\}$ of length $n$.

  Then, by plugging in
  \[ V_{t}(z) = \frac{z (1 - \sqrt{1 - 4z^{2}})^{r} (1 + \sqrt{1 - 4z^{2}})}{2^{r+1}
      \sqrt{1  - 4z^{2}}} \]
  and the formula for $V(z)$ into
  \[ \frac{\partial }{\partial t}F(z,t,1) \big|_{t=1} = \frac{z^{r} (1 - z)^{2} (1 -
      V(z))}{(1 - 2z)^{2}} - \frac{V_{t}(z)}{1 - 2z},  \]
  we have an explicit representation of the generating function for the expected number of
  $r$-ascents before normalization. Extracting the coefficients by means of singularity
  analysis (the location $z = \pm 1/2$ of the dominant singularities is known from above)
  and then dividing by~\eqref{eq:meanders:dyck-number-of}
  yields~\eqref{eq:meanders-expectation:dyck}.

  For computing the variance, we proceed similarly: we determine the asymptotic behavior of
  the second factorial moment $\E(M_{n,r} (M_{n,r} - 1))$ by extracting the coefficients
  of $\frac{\partial^{2}}{\partial t^{2}} F(z,t,1) |_{t=1}$ and normalizing the result
  by dividing by~\eqref{eq:meanders:dyck-number-of}. Then,
  \eqref{eq:meanders-variance:dyck} follows from
  \[ \V M_{n,r} = \E(M_{n,r}(M_{n,r} - 1)) + \E M_{n,r} - (\E M_{n,r})^{2}. \qedhere \]
\end{proof}
\begin{proposition}\label{prop:meanders:motzkin}
  The expected number of
  $r$-ascents in the \L ukasiewicz meanders of length $n$ associated to $\mathcal{S} =
  \{-1, 0, 1\}$ as
  well as the corresponding variance grow with $n\to\infty$ according to the asymptotic
  expansions
  \begin{multline}\label{eq:meanders-expectation:motzkin}
    \E M_{n,r} = \frac{2^{r}}{3^{r+2}} n + \frac{\sqrt{3 \pi} (r - 4) 2^{r-2}}{3^{r+2}}
    n^{1/2} - (3r^{2} - r - 96) \frac{2^{r-4}}{3^{r+2}} \nopagebreak \\  + \frac{\sqrt{3\pi} (r-4)
      2^{r-6}}{3^{r}} n^{-1/2} + O(n^{-1})
  \end{multline}
  and
  \begin{multline}\label{eq:meanders-variance:motzkin}
    \V M_{n,r} = \frac{3^{r+2} 2^{r+4} - 2^{2r} (3r^{2}(\pi - 2) - 8r(3\pi - 10) +48\pi -
      144)}{16\cdot 3^{2r+4}} n \\ + \frac{\sqrt{3\pi} (72 (r-4) 6^{r} - 2^{2r}(3r^{3} -
      9r^{2} - 28r - 32))}{32\cdot 3^{2r+4}} n^{1/2} + O(1).
  \end{multline}
\end{proposition}
\begin{proof}
  The asymptotic expansions for expectation and variance can be obtained with an analogous
  approach as in the proof of Proposition~\ref{prop:meanders:dyck}. In this case, we have
  \[ V(z) = \frac{1 - z - \sqrt{1 - 2z - 3z^{2}}}{2z},  \]
  and the dominant singularity of $F(z,1,1)$ (as well as for the corresponding
  derivatives with respect to $t$) is located at $z = 1/3$.
\end{proof}

\bibliographystyle{bibstyle/amsplainurl}
\bibliography{bib/cheub.bib}

\providecommand{\Submitted}{Submitted} \providecommand{\availableat}{ available
  at } \providecommand{\alsoavailableat}{ also available at }
  \providecommand{\evavailableat}{earlier version available at }
  \providecommand{\toappearin}{To appear in } \providecommand{\toappear}{to
  appear} \providecommand{\inpreparation}{in preparation}
  \providecommand{\doi}[1]{\href{http://dx.doi.org/#1}{\path{doi:#1}}}
  \providecommand{\lowercaseforams}{}
  \providecommand{\etc}{\emph{etc.}}\def\cprime{$'$}
\providecommand{\bysame}{\leavevmode\hbox to3em{\hrulefill}\thinspace}
\providecommand{\MR}{\relax\ifhmode\unskip\space\fi MR }
\providecommand{\MRhref}[2]{%
  \href{http://www.ams.org/mathscinet-getitem?mr=#1}{#2}
}
\providecommand{\href}[2]{#2}
\begin{thebibliography}{10}

\bibitem{Banderier-et-al:2002:generating-functions-trees}
Cyril Banderier, Mireille Bousquet-M\'elou, Alain Denise, Philippe Flajolet,
  Dani\`ele Gardy, and Dominique Gouyou-Beauchamps,
  \href{https://doi.org/10.1016/S0012-365X(01)00250-3}{\emph{Generating
  functions for generating trees}}, Discrete Math. \textbf{246} (2002),
  no.~1-3, 29--55, Formal power series and algebraic combinatorics (Barcelona,
  1999). \MR{1884885}

\bibitem{Banderier-Flajolet:2002:lat-path}
Cyril Banderier and Philippe Flajolet,
  \href{http://dx.doi.org/10.1016/S0304-3975(02)00007-5}{\emph{Basic analytic
  combinatorics of directed lattice paths}}, Theoret. Comput. Sci. \textbf{281}
  (2002), no.~1--2, 37--80. \MR{1909568 (2003g:05006)}

\bibitem{Flajolet-Sedgewick:ta:analy}
Philippe Flajolet and Robert Sedgewick,
  \href{http://dx.doi.org/10.1017/CBO9780511801655}{\emph{Analytic
  combinatorics}}, Cambridge University Press, Cambridge, 2009.

\bibitem{Graham-Knuth-Patashnik:1994}
Ronald~L. Graham, Donald~E. Knuth, and Oren Patashnik, \emph{Concrete
  mathematics. {A} foundation for computer science}, second ed.,
  Addison-Wesley, 1994.

\bibitem{Hackl-Heuberger-Krenn:2016:asy-sagemath}
Benjamin Hackl, Clemens Heuberger, and Daniel Krenn, \emph{Asymptotic
  expansions in {SageMath}}, \url{http://trac.sagemath.org/17601}, 2015, module
  in \href{http://www.sagemath.org/}{SageMath 6.10}.

\bibitem{Kangro-Pourmoradnasseri-Theis:2016:ascents-dispersed-dyck}
Kairi Kangro, Mozhgan Pourmoradnasseri, and Dirk~Oliver Theis,
  \href{https://arxiv.org/abs/1603.01422}{\emph{Short note on the number of
  1-ascents in dispersed {D}yck paths}}, arXiv:1603.01422 [math.CO], 2016.

\bibitem{Kaup-Kaup:1983:holomorphic-several-variables}
Ludger Kaup and Burchard Kaup,
  \href{https://doi.org/10.1515/9783110838350}{\emph{Holomorphic functions of
  several variables}}, De Gruyter Studies in Mathematics, vol.~3, Walter de
  Gruyter \& Co., Berlin, 1983, An introduction to the fundamental theory, With
  the assistance of Gottfried Barthel, Translated from the German by Michael
  Bridgland. \MR{716497}

\bibitem{Lothaire:1997:comb-words}
M.~Lothaire,
  \href{http://dx.doi.org/10.1017/CBO9780511566097}{\emph{Combinatorics on
  words}}, Cambridge Mathematical Library, Cambridge University Press,
  Cambridge, 1997, With a foreword by Roger Lyndon and a preface by Dominique
  Perrin, Corrected reprint of the 1983 original, with a new preface by Perrin.
  \MR{1475463}

\bibitem{Prodinger:2003:kernel-examples}
Helmut Prodinger, \emph{The kernel method: a collection of examples}, S\'em.
  Lothar. Combin. \textbf{50} (2003/04), Art. B50f, 19. \MR{2079850}

\bibitem{Bona:Prodinger:2015:analyt}
Helmut Prodinger, \emph{Analytic methods}, Handbook of Enumerative
  Combinatorics (Mikl{\'o}s B{\'o}na, ed.), Discrete Mathematics and its
  Applications (Boca Raton), Chapman \& Hall/CRC, Boca Raton, FL, 2015,
  pp.~173--252. \MR{3408702}

\bibitem{SageMath:2017:8.1}
The~SageMath Developers, \emph{{SageMath} {M}athematics {S}oftware ({V}ersion
  8.1)}, 2017, \url{http://www.sagemath.org}.

\end{thebibliography}

\end{document}